\documentclass[11pt]{amsart}

\usepackage{epigamath}

\usepackage[english]{babel}

\numberwithin{equation}{section}

\usepackage[shortlabels]{enumitem}

\usepackage{amssymb}
\usepackage{amsmath, amscd}
\usepackage{amsthm}
\usepackage{graphicx}
\usepackage{hyperref}

\usepackage[curve,all]{xy}
 \xyoption{curve}
 \xyoption{line}
\xyoption{arc}

\usepackage{color}
\usepackage{amsthm}

\usepackage{extarrows}
\usepackage{array}
\usepackage{float}

\usepackage{mathtools}

\newtheorem{theorem}{Theorem}[section]
\newtheorem{lemma}[theorem]{Lemma}
\newtheorem{corollary}[theorem]{Corollary}
\newtheorem{proposition}[theorem]{Proposition}

\theoremstyle{definition}

\theoremstyle{remark}
\newtheorem{remark}[theorem]{Remark}
\newtheorem{example}[theorem]{Example}

\newcommand{\calO}{\mathcal{O}}

\newcommand{\PP}{\mathbb{P}}
\newcommand{\Z}{\mathbb{Z}}
\newcommand{\Jac}{\mathrm{Jac}}

\newcommand{\bbZ}{\mathbb{Z}}

\DeclareMathOperator{\Aut}{Aut}
\DeclareMathOperator{\Pic}{Pic}
\DeclareMathOperator{\Num}{Num}

\def\I{{\mathrm{I}}}
\def\II{{\mathrm{II}}}
\def\III{{\mathrm{III}}}
\def\IV{{\mathrm{IV}}}

\setlist[enumerate,1]{label={\rm(\roman*)}, ref={\rm\roman*}} 

\newcounter{foo}
\makeatletter
\newcommand{\otherlabel}[2]{\protected@edef\@currentlabel{#2}\label{#1}}
\makeatother

\def\lra{\longrightarrow}

\newcommand{\supth}[1]{\ensuremath{#1^{\mathrm{th}}}}

\EpigaVolumeYear{8}{2024} \EpigaArticleNr{8} \ReceivedOn{May 31, 2023}
\InFinalFormOn{October 9, 2023}
\AcceptedOn{January 12, 2024}

\title{Normal forms for quasi-elliptic Enriques surfaces and applications}
\titlemark{Normal forms for quasi-elliptic Enriques surfaces}

\author{Toshiyuki Katsura}
\address{Graduate School of Mathematical Sciences, The University of Tokyo, 
Meguro-ku, Tokyo 153-8914, Japan}
\email{tkatsura@g.ecc.u-tokyo.ac.jp}

\author{Matthias Sch\"utt}
\address{Institut f\"ur Algebraische Geometrie, 
Leibniz Universit\"at  Hannover, Welfengarten 1, 30167 Hannover, Germany\\
Riemann Center for Geometry and Physics,
Leibniz Universit\"at Hannover, Appelstrasse 2, 30167 Hannover, Germany}
\email{schuett@math.uni-hannover.de}

\authormark{T.~Katsura and M.~Sch\"utt}

\AbstractInEnglish{We work out normal forms for quasi-elliptic Enriques surfaces and give several applications.  These include torsors and numerically trivial automorphisms, but our main application is the completion of the classification of Enriques surfaces with finite automorphism groups started by Kond\=o, Nikulin, Martin and Katsura--Kond\=o--Martin.}



\acknowledgement{Research of the first author is partially supported by JSPS Grant-in-Aid for Scientific Research (C) No.23K03066.}

\dedication{Dedicated to Shigeyuki Kond\=o on the occasion of his 65th birthday}

\begin{document}


\maketitle

\begin{prelims}

\DisplayAbstractInEnglish

\bigskip

\DisplayKeyWords

\medskip

\DisplayMSCclass

\end{prelims}


\newpage

\setcounter{tocdepth}{1}

\tableofcontents


\section{Introduction}

Enriques surfaces are an important class of algebraic surfaces.  Their historical fame is due to the property that outside characteristic $2$, their irregularity and geometric genus both vanish although they are not rational, which subsequently motivated Castelnuovo's rationality criterion.

Characteristic $2$ is special in this context because it features different types of Enriques surfaces; see \cite{BM}.  This is our motivation to focus onto a yet more special class of Enriques surfaces, namely those admitting a quasi-elliptic fibration; here the general fibre is a cuspidal cubic.

Despite their special nature, quasi-elliptic Enriques surfaces are central for our understanding of several aspects of the theory of Enriques surfaces over an algebraically closed field $k$ of characteristic $2$.  We are therefore convinced that our main result, stated below, will have many interesting consequences, a few of which we will work out in this paper as well.

\begin{theorem}
\label{thm}
Any quasi-elliptic Enriques surface is given by an affine equation of the following form, where each polynomial $a_i\in k[t]$ is of degree at most $i$:
\begin{enumerate}
\item\label{thm-1}
If $S$ is classical, then 
\begin{eqnarray}
\label{eq:classical}
S: \;\; y^2 + t^2 a_1 y = tx^4 + t^3 a_0 x^2 + t^3 a_2 x + t^3 (1+t)^4.
\end{eqnarray}
\item\label{thm-2}
If $S$ is supersingular, then  
\begin{eqnarray}
\label{eq:supersingular}
S: \;\;y^2  + t^4 a_1 y = tx^4 + t^5 a_0 x^2 + t^6 a_2 x + t^3.
\end{eqnarray}
\end{enumerate}
In each case, we only require that $(a_1, a_2)\not\equiv (0,0)$.
\end{theorem}

One can also give a combined equation covering both cases to view the supersingular case as a specialization of the classical one; see \eqref{eq:general'} below or \eqref{eq:general} in Theorem~\ref{rem:both}.  These equations are very useful for explicit computations, similarly to the Weierstrass form of an elliptic surface with section.  We shall demonstrate this with three major applications.

Our first application concerns the Enriques torsors above a given rational quasi-elliptic surface $X$. 

\begin{theorem}
\label{thm2}
A general rational quasi-elliptic surface $X$ with section admits an irreducible $4$-dimensional family of torsors of classical Enriques surfaces and an irreducible $3$-dimensional family of torsors of supersingular Enriques surfaces.

More precisely, given any $X$, the families of Enriques torsors have dimension $1$ $($resp.\ $2)$ less than in the general case if and only if $X$ has only two $($resp.\ one$)$ reducible fibre, except for the supersingular Enriques surfaces with a single reducible fibre of type $\I_4^*$, which depend on two moduli.

Explicitly, if $X$ is given by the Weierstrass form
\[
X: \;\; y^2 = x^3 + t(a_1^2+ta_0^2) \, x + t a_2^2, \quad a_0\in k[t], \;\; \deg(a_i)\leq i,
\]
then the Enriques torsors are given by 
\begin{eqnarray}
\label{eq:general'}
y^2 + g_2^2 a_1 y = tx^4 + t g_2^2 a_0 x^2 + g_2^3 a_2 x + t^3 c_1^4,
\end{eqnarray}
where $c_1, g_2\in k[t]\setminus\{0\}$ have degrees $\deg(c_1)\leq 1$ and $\deg(g_2)\leq 2$ and $c_1\nmid g_2$ $($and $\deg(g_2)=2$ if $\deg(c_1)=0)$.
\end{theorem}

In comparison, in characteristic zero, there are $2$-dimensional families of torsors of Enriques surfaces above a given rational elliptic surface (but they are not necessarily irreducible), and the same holds true for very general rational elliptic surfaces in any characteristic for moduli-dimension reasons.  Theorem~\ref{thm2} thus highlights once again how special quasi-elliptic Enriques surfaces are.

Theorem~\ref{thm2} enables us to translate Ito's classification of rational quasi-elliptic surfaces into an explicit classification of the Enriques torsors; see Table~\ref{tab2}.  As a key application, this will also put us in the position to complete the classification of Enriques surfaces with finite automorphism groups, our second application.  After the work of Kond\=o \cite{Kondo-aut}, Nikulin \cite{Nikulin} and Martin \cite{Martin},  only the cases of classical and supersingular Enriques surfaces are left, where the possible graphs $\Gamma$ of smooth rational curves have been computed in \cite{KKM}, but the automorphism groups and the moduli involved have not been determined completely yet.  Using Theorem~\ref{thm}, we can remedy this with our second main result. 

\begin{theorem}
\label{thm3}
Let $S$ be an Enriques surface with finite automorphism group.  Then $S$ appears in \cite{Martin}, for singular Enriques surfaces, or in \cite{KKM}, or, for type $\Gamma=\tilde E_6+\tilde A_2$, in the family \eqref{c4} from Theorem~\ref{thm:c-eqn}, for classical and supersingular Enriques surfaces.

In particular, classical or supersingular Enriques surfaces with finite automorphism group form irreducible families, uniquely depending on $\Gamma$, of the dimension and with the automorphism group stated in \cite{KKM}.
\end{theorem}

Explicit equations are given in Theorems~\ref{thm:c-eqn} and~\ref{thm:s-eqn}.  It also follows from Theorem~\ref{thm3} that the subgroups of cohomologically or numerically trivial automorphisms are as stated in \cite{KKM}.  By \cite{DM}, this leaves open only the existence of Enriques surfaces with a cohomologically trivial automorphism of order $3$.  Using Theorems~\ref{thm} and~\ref{thm2}, we can also solve this. 

\begin{theorem}
\label{thm:ct}
Let $S$ be an Enriques surface (in any characteristic) with a cohomologically trivial automorphism of order $3$.  Then $S$ is a supersingular Enriques surface in characteristic $2$ and belongs to the following family:
\begin{eqnarray}
\label{eq:ct}
S:\;\; y^2 = tx^4 + \alpha t^5x^2 + t^7 x + t^3 \quad (\alpha\in k); 
\end{eqnarray}
the cohomologically trivial automorphism of order $3$ is given by $(x,y,t) \mapsto (\zeta^2 x, y, \zeta t)$, where $\zeta$ denotes a primitive third root of unity.
\end{theorem}

As a consequence, the full picture of possible groups of numerically trivial automorphisms of  Enriques surfaces follows from \cite{DM,DM2,KKM} combined with Theorems~\ref{thm3} and~\ref{thm:ct} (and similarly for cohomologically trivial automorphisms). 

\begin{corollary}
\label{cor:ct}
A group $G$ appears as  group of numerically trivial automorphisms of some Enriques surface over $k$ if and only if
\begin{enumerate}
\item\label{c:ct-1}
$G\in\{\{1\}, \Z/2\Z, \Z/4\Z\}$ if char$(k)\neq 2$;
\item\label{c:ct-2}
$G\in\{\{1\},\Z/2\Z\}$ if char$(k)=2$ and the Enriques surface is singular;
\item\label{c:ct-3}
$G\in\{\{1\}, \Z/2\Z, (\Z/2\Z)^2\}$ if char$(k)=2$ and the Enriques surface is classical;
\item\label{c:ct-4}
$G\in \{\{1\}, \bbZ/2\bbZ, \bbZ/3\bbZ, \bbZ/5\bbZ, \bbZ/7\bbZ, \bbZ/11\bbZ, Q_8 \}$ if char$(k)=2$ and the Enriques surface is supersingular.
\end{enumerate}
\end{corollary}

One can also apply our results to the study of maximal root types supported on Enriques surfaces, \textit{i.e.}, rank $9$ root lattices whose vertices correspond to smooth rational curves.  In \cite{S-hom2} there is given a complete classification of the maximal root types for singular and classical Enriques surfaces.  This can now be complemented for many types on supersingular Enriques surfaces.  For another direction of applications of our methods, see Remark~\ref{rem:appl}.

\subsection*{Organization of the paper}

After reviewing, in Section~\ref{eq:def}, the basics of Enriques surfaces needed for our work, we start by developing an equation for nodal Enriques surfaces valid in any characteristic using Riemann--Roch (Section~\ref{eq:def}).  Then we specialize to the quasi-elliptic setting.  We start out by reviewing Queen's equations in Section~\ref{s:Queen}.  Applying the results from Section~\ref{eq:def}, we obtain a first working version of the desired normal form in \eqref{eq:unique} which already resembles the Weierstrass form of an elliptic curve to some extent.

The ensuing Sections~\ref{s:rJ}--\ref{s:pf} are devoted to a thorough analysis of the impact of the condition of $S$ being an Enriques surface on the polynomials occurring as coefficients \eqref{eq:unique}.  We first introduce the relative Jacobian in Section~\ref{s:rJ}; this is a rational quasi-elliptic surface sharing the same reducible fibres with the quasi-elliptic Enriques surface $S$.  Notably this implies high divisibility conditions for the discriminant and subsequently also for the coefficients (see Lemma~\ref{lem:div}).  In turn, these translate into singularities of the resulting normal form which will crucially impact the canonical divisor (see \eqref{eq:model}).  In Section~\ref{s:ade}, we concentrate on ADE singularities and show that they exclusively occur on simple fibres (see Corollary~\ref{lem:simple}).  Section~\ref{s:res} covers more complicated singularities; it develops an explicit resolution of the singularities (similar in spirit to Tate's algorithm for elliptic curves) which in particular shows that the underlying fibres are multiple unless the degree of \eqref{eq:model} can be reduced.  The explicit nature of the resolution is used in Section~\ref{s:canon} to detect the impact on the canonical divisor.  Section~\ref{s:pf} then collects all geometric and topological information needed to prove Theorem~\ref{thm}.

The final five sections are concerned with the applications given above, namely to torsors, numerically or cohomologically trivial automorphisms and finite automorphism groups.

\subsection*{Acknowledgements}
We are very grateful to the referee for the thorough comments which greatly helped us improve the paper.  We thank S.\ Kond\=o for helpful discussions and I.\ Dolgachev and G.\ Martin for interesting comments.

\section{Basics on Enriques surfaces}
\label{s:basics}

Let $S$ be an Enriques surface, \textit{i.e.}~a smooth algebraic surface with
\[
b_2(S) = 10, \quad K_S\equiv 0
\]
regardless of the characteristic.  Outside characteristic $2$, Enriques surfaces form an irreducible $10$-dimensional family, but in characteristic $2$ there are three classes of Enriques surfaces by \cite{BM}, which depend on the Picard scheme $\Pic^\tau(S)$ as follows:
$$
\begin{array}{lcl}
\text{classical:} && \Pic^\tau(S) = \Z/2\Z,\\
\text{singular:} && \Pic^\tau(S) = \mu_2,\\
\text{supersingular:} && \Pic^\tau(S) = \alpha_2. 
\end{array}
$$
Both classical and singular Enriques surfaces form irreducible $10$-dimensional families; their closures intersect in the $9$-dimensional supersingular locus (\textit{cf.} \cite{Liedtke}).  Singular Enriques surfaces behave as in characteristic zero in the sense that they are quotients of K3 surfaces by free involutions.  Hence many ideas and results carry over; for instance, the classification of singular Enriques surfaces with finite automorphism group is completely known thanks to \cite{Martin}, while there are a few open questions for the other types of Enriques surfaces which we will answer in this paper.

By \cite{Illusie}, we know that $\rho(S)=10$ and $\Num(S)\cong U \oplus E_8$.  In particular, $\Num(S)$ represents zero non-trivially, and by Riemann--Roch, $S$ admits a genus $1$ fibration
\begin{eqnarray}
\label{eq:fibr}
f\colon S \lra \PP^1.
\end{eqnarray}
Outside characteristic $2$, this comes with two multiple fibres (thus with no section), but in characteristic $2$, the class of the Enriques surface $S$ is given by the multiple fibre(s) (of multiplicity $2$) as follows by \cite[Theorems~5.7.5 and 5.7.6]{CD}:
\begin{table}[ht!]
\begin{tabular}{lcl}
classical: && two multiple fibres, both  ordinary or  additive, \\
singular: && one multiple fibre,  ordinary or multiplicative, \\
supersingular: && one multiple fibre,  supersingular or additive.  
\end{tabular}
\end{table}

\noindent
Note that any genus~$1$ fibration on a singular Enriques surface is elliptic.

Many properties of an Enriques surface are governed by the question of whether it contains a smooth rational curve $C$ (often called a nodal curve or $(-2)$-curve, for $C^2=-2$).  Nodal Enriques surfaces have codimension $1$ in moduli. By \cite{Cossec, Lang}, on every nodal Enriques surface, there exists a $(-2)$-curve $C$ that appears as a bisection of some genus~$1$ fibration \eqref{eq:fibr}.  In the next section, we will start by briefly considering nodal Enriques surfaces in full generality.  These considerations will apply subsequently to quasi-elliptic Enriques surfaces in characteristic $2$ because these always contain a nodal curve, namely the curve of cusps.

\section{The defining equation} 
\label{eq:def}

In this section we let $k$ denote an arbitrary algebraically closed field.  Let $S$ be a nodal Enriques surface, and fix a genus~$1$ fibration \eqref{eq:fibr} with nodal bisection $C$.  Let $t$ be a coordinate of ${\bf A}^1 \subset {\bf P}^1$. Then, $t$ has a pole of order~$1$ at the point $P_{\infty}$ at infinity.  We may assume that $f \colon S \rightarrow \PP^1$ has a multiple fibre at $P_{\infty}$.  We set $f^{-1}(P_{\infty}) = 2F_{\infty}$. Then, we have
$$
     C^2 = -2, \quad (C\cdot F_{\infty})= 1,\quad F_{\infty}^2 = 0.
$$
By a suitable M\"obius translation, we may assume that the fibre defined by $t=0$ is simple and irreducible.

We have the following vanishing theorem.

\begin{theorem}[Cossec, Dolgachev and Liedtke, \textit{cf.} \protect{\cite[Theorem 2.1.15]{CDL}}]\label{vanishing}
Let $S$ be an Enriques surface defined over $k$ and $D$ be a big and nef divisor on $S$. Then we have ${\rm H}^{1}(S, \calO_{S}(D)) = 0$.
\end{theorem}

\begin{lemma}\label{nef-and-big}
Let $a,b\in\Z$. Then $D=aC+bF_\infty$ is big and nef if and only if $a>0$ and $b\geq 2a$.
\end{lemma}

\begin{proof}
If $D$ is nef, then intersecting with $F_\infty$ gives $a>0$ and intersecting with $C$ implies $b\geq 2a$.
 
Conversely, $D^2=2a(b-a)>0$ directly shows that $D$ is big.  For any irreducible component $\Theta$ of $F_{\infty}$, we have $(F_{\infty}\cdot \Theta) = 0$. Therefore, we have $(D\cdot \Theta) \geq 0$. Moreover, $(D\cdot C) = -2a + b \geq 0$ by assumption.  But then, this gives $(D\cdot C') \geq 0$ for any irreducible curve $C'\subset S$.  Hence $D$ is nef, as claimed.
\end{proof}

\begin{lemma}\label{dimension}
Let $a,b\in\Z_{>0}$ with $b\geq 2a$. Then $D=aC+bF$ satisfies $\dim L(D)=1+ab-a^2$.
\end{lemma}

\begin{proof}
Since we have $\chi ({\calO}_S) = 1$, by the Riemann--Roch theorem, we have
$$
\chi ({\calO}_S(D)) 
= (D^2 - D \cdot K_S)/2 + 1 = a(b-a)+1.
$$
From Theorem~\ref{vanishing}, Lemma~\ref{nef-and-big}  and the Serre duality theorem, we deduce
$$
{\rm H}^i(S, {\calO}_S(D)) = 0\quad  (i = 1, 2) \quad \mbox{for}~ n \geq 2.
$$
Therefore, we obtain $\dim L(D) = ab-a^2+1$, as claimed.
\end{proof}

We will use Lemma~\ref{dimension} to calculate a defining equation of $S$.
It follows from Lemma~\ref{dimension} that there exist
\[
x \in L(C+3F_\infty), \quad y \in L(2C+4F_\infty)
\]
such that
$\{ 1, t, t^2 , x, y\}$ is a basis of $L(2C + 4F_{\infty})$.

Consider the vector space $L(4C + 16 F_{\infty})$.  It follows from the choice of the functions $x,y,t$ that the following 50 functions are contained in $L(4C + 16 F_{\infty})$:
$$
\begin{array}{l}
t^{i}~(0\leq i\leq 8),~t^ix~(0\leq i\leq 6),~t^ix^2~(0\leq i\leq 5),\\
t^ix^3~(0\leq i\leq 3), ~t^ix^4~(0\leq i\leq 2),\\
t^iy~(0\leq i\leq 6),~t^iy^2~(0\leq i\leq 4),~t^ixy~(0\leq i\leq 4),~t^ix^2y~(0\leq i\leq 3). 
\end{array}
$$
On the other hand, by Lemma~\ref{dimension}, we have $\dim L(4C + 16 F_{\infty}) = 49$.  Therefore, these 50 functions are linearly dependent over $k$.  A non-trivial linear relation between these functions is expressed as
\begin{equation}\label{dependent}
h_4(t)y^2 + k_3(t)x^2y + k_4(t)xy +h_6(t)y 
 = g_2(t)x^4 + g_3(t)x^3 + g_5(t)x^2 + g_{6}(t)x + g_{8}(t)
\end{equation}
with indices indicating the degree of the respective polynomial in $k[t]$.

\begin{theorem}\label{equation}
Any nodal Enriques surface  is birationally expressed
by  \eqref{dependent}.
\end{theorem}

\begin{proof}
Consider the generic fibre $E$ of $f \colon S \rightarrow {\bf P}^1$.  This is a curve of genus $1$ over $k(t)$, and the bisection $C$ gives a point $P$ of degree $2$ on the curve $E$. By the Riemann--Roch theorem for curves, we have $\dim L(nP) = 2n$ for $n \geq 1$.  By the consideration above, we see that $\{ 1, x\}$ is a basis of $L(P)$, and $\{ 1, x, x^2, y\}$ is a basis of $L(2P)$.  It is easy to see that $1, x, x^2, x^3, x^4, y, y^2, xy, x^2y$ are elements of $L(4P)$.  Since $\dim L(4P) = 8$, these nine elements are linearly dependent over $k(t)$. Therefore, Equation \eqref{dependent} is nothing but the desired linear relation over $k(t)$.

We continue to argue with $2P$.  Since this is very ample on $E$, $E$ is embedded into ${\mathbb P}^3$ via $[X_0,X_1,X_2,X_3] = [1,x,x^2,y]$.  But then the image lies on the conic $\{X_0X_2=X_1^2\}$, so we can eliminate $X_2$ and obtain \eqref{dependent} as affine equation over $k(t)$.  This means that our Enriques surface is birationally equivalent to the surface defined by Equation \eqref{dependent}.  This proves the theorem.
\end{proof}

In the next section, we will specialize to the quasi-elliptic setting in characteristic $2$ to derive a much more convenient equation (as ultimately displayed in Theorem~\ref{thm}) which will also lend itself to several applications.

\section{Queen-type equations}
\label{s:Queen}

We now specialize to the situation in characteristic $2$ where the genus~$1$ fibration is quasi-elliptic and the nodal curve $C$ is the curve of cusps.  We will use this setting to simplify  Equation \eqref{dependent}.  Our results should be compared to those of Queen; see \cite{Q1, Q2}.  The main difference is that Queen works over a field, so he can simplify further, while we prefer to preserve some polynomial shape with good control over the degrees involved.
 
We let $K = k(t)$ and consider the generic fibre $E$ of the quasi-elliptic fibration.  This is endowed with the degree $2$ point $P$ at infinity corresponding to the curve of cusps, so we can define functions $x, y$ in complete analogy with Section~\ref{eq:def}.  Consider the degree $2$ extension $K(E)/K(x)$.  In what follows, we distinguish whether this extension is separable or not.

\subsection{Inseparable case}
 
As above, $\{ 1, x, x^2, y\}$ is a basis of $L(2P)$, and we have $K(E) = K(x, y)$.  Since $K(E)/K(x)$ is a purely inseparable extension of degree $2$, we see that $y^2 \in K(x)$.  On the other hand, we have $y^2 \in L(4P)$ and
$$
K + Kx + Kx^2 + Kx^3 + Kx^4 = K(x) \cap L(4P).
$$
Therefore, there exist elements $a, b, c, d, e \in K$ such that
$$
y^2 = ax^4 + bx^3 + cx^2 + dx + e.
$$
Suppose that $b$ is not zero.  Then, differentiating this equation with respect to $x$, we have the singular locus of $E$ defined by $bx^2 + d = 0$, which contradicts our assumption that the infinite point $P$ is the cusp. Thus $b=0$ and
\begin{equation}\label{caseI}
h_4(t)y^2 = g_2(t)x^4 + g_5(t)x^2 + g_{6}(t)x + g_{8}(t).
\end{equation}

\subsection{Separable case}
\label{ss:separable}

We consider $L(4P)$. Then, by the Riemann--Roch theorem, we have $\dim L(4P) = 8$. Since $1$, $x$, $x^2$, $x^3$, $x^4$, $y$, $xy$, $x^2y$ and $y^2$ are elements of $L(4P)$, we have a linear relation
\begin{equation}\label{-}
y^2 + (ax^2 + bx + c)y =  dx^4 + ex^3 + fx^2 + gx + h
\end{equation}
with $a, b, c, d, e, f, g, h \in K$. Note that by considering the pole order at $P$, we see that the coefficients of $y^2$ and of $x^4$ are non-zero, so we can take the coefficient of $y^{2}$ to be $1$ and assume $d \neq 0$.  By the change of coordinates $X = 1/x$ and $Y = y/x^2$, we have
\begin{equation}\label{+}
  Y^2 + (a+ bX + cX^2)Y =  d + eX + fX^2 + gX^3 + hX^4.
\end{equation}
By our assumption, the point $P$ of degree $2$ defined by $X = 0$ is the cusp singularity.  Therefore, differentiating with respect to $X$ and $Y$, we infer that $X = 0$ must be a solution of the equations
$$
a + bX + cX^2 = 0,
\quad
 bY = e + gX^2.
$$
Therefore, we have $a = 0$ and $bY = e$.  Suppose $b \neq 0$. Then, we have $Y = e/b$. Substituting these results in Equation \eqref{+}, we have $(e/b)^2 = d$, and Equation \eqref{-} becomes
$$
     y^2 + (bx + c)y =  (e/b)^2 x^4 + ex^3 + fx^2 + gx + h.
$$
By the change of coordinates $\tilde{y} = y + (e/b)x^2$, this equation is converted to a cubic in $x$ and $y$.  By inspection, we see that it attains a section at infinity.  In particular, this quasi-elliptic surface cannot have multiple fibres, which contradicts our assumption.  Hence, we see that $b = 0$ and $e = 0$, and our equation becomes
$$
 {y}^2 + c{y} =  dx^4 + fx^2 + gx + h.
$$
Since $K(x, y)/K(x)$ is separable, we have $c \neq 0$.  Applying this calculation to Equation \eqref{dependent}, we obtain
\begin{equation}\label{caseII}
h_4(t)y^2 + h_6(t)y  = g_2(t)x^4  + g_5(t)x^2 + g_{6}(t)x + g_{8}(t).
\end{equation}
Note that this contains \eqref{caseI} as a subfamily, though in what follows the two equations will sometimes display quite different behaviour.

\section{General normal form}
\label{ss:charts}

We aim to convert Equations \eqref{caseI} and \eqref{caseII} alike to a general normal form. To this end, we multiply both sides of \eqref{caseII} with $h_{4}$.  Replacing $h_{4}y$ with $y$, we obtain the equation
\begin{equation*}\label{h'4}
y^2 + h_6y  = h_4g_2x^4  + h_4g_5x^2 + h_4g_{6}x + h_4g_{8}.
\end{equation*}
Writing $h_4g_2 = h_3^2+th_2^2$, we can translate $y$ by $h_3x^2$ to get
\[
y^2 + h_6y  = th_2^2x^4  + (h_4g_5+h_3h_6)x^2 + h_4g_{6}x + h_4g_{8}.
\]
Dividing $x$ and $y$ by $h_2$, this leads to
\begin{eqnarray}
\label{eq:normal}
y^2 + h_2h_6y  = tx^4  + (h_4g_5+h_3h_6)x^2 + h_2h_4g_{6}x + h_2^2h_4g_{8}.
\end{eqnarray}
We could continue by analysing this equation (for instance the special fibres at the zeroes of $h_2$, or the purported multiple fibre at $\infty$), but for the sake of a unified treatment, we will content ourselves with the overall shape of a complete model.  To this end, we attach the weights $9$ to $y$ and $4$ to $x$ and homogenize \eqref{eq:normal} as an equation of degree $18$ in $\PP[1,1,4,9]$:
\begin{eqnarray}
\label{eq:homog}
y^2 + a_9y  = stx^4  + a_{10}x^2 + a_{14}x + a_{18}.
\end{eqnarray}
Here and in what follows, the $a_i$ will be regarded as homogeneous polynomials in $k[s,t]$ of degree given exactly by the index, though we will take the liberty to suppress $s$ from notation for ease of presentation.  If necessary, a complete model of the surface can be described by four affine charts, namely
\begin{enumerate}[(1)]
\item
the chart in \eqref{eq:normal} and those obtained from it as follows:
\item
the chart with affine coordinates $X=1/x, Y=y/x^2, t$ as in Section~\ref{ss:separable};
\item
the standard chart at $t=\infty$ with affine coordinates $s=1/t, u=x/t^4, v=y/t^9$;
\item
the chart analogous to the second one with coordinates $U=1/u, V=v/u^2, s$.
\end{enumerate}
Note that the shape of \eqref{eq:homog}, including the coefficient of $x^4$, is preserved by the admissible coordinate transformations
\begin{eqnarray}
\label{eq:admissible}
(x,y) \longmapsto (x+b_4, y +b_5x+b_9), 
\end{eqnarray}
where $b_i\in k[t]$ is of degree $i$.  (Additionally, we could also rescale coordinates $(x,y)\mapsto (ux,u^2y)$ for $u\in k^\times$, which will appear later in Lemma~\ref{lem:iso}).

\begin{lemma}
\label{lem:unique}
There is an admissible transformation converting \eqref{eq:homog} to
\begin{eqnarray}
\label{eq:unique}
S: \;\; y^2 + a_9y  = tx^4  + ta_{4}^2x^2 + a_{14}x + t^3a_{3}^4.
\end{eqnarray}
\end{lemma}

\begin{remark}\label{rem:choice}\leavevmode
\begin{enumerate}
\item\label{r:c-1}
The shape of \eqref{eq:unique} is symmetric in $t$ and the suppressed homogenizing variable $s$.
This will be quite useful later when we locate the multiple fibres (in the classical case) at $0$ and $\infty$; 
\textit{cf.} \eqref{eq:classical}.
\item\label{r:c-2}
Note that we have not made any assumption on the special fibres at $t$ (or $\infty$) yet,
so \eqref{eq:unique} is universally valid locally at any given fibre
(but  with coefficients depending on the choice of fibre).
\end{enumerate}
\end{remark}

\begin{proof}
The proof of Lemma~\ref{lem:unique} relies on the following easy general  result. 

\begin{lemma}
\label{lem:aux}
Let $k$ be an algebraically closed field.  Let $n\in\mathbb N$ and $h_1, \hdots, h_n \in k[z_1,\hdots,z_n]$ be such that for each $i$, 
\[
h_i = z_i^{d_i} + (\text{terms of total degree } < d_i).
\]
Then there is a common zero of all $h_i$ in $k^n$.
\end{lemma}

\begin{proof}[Proof of Lemma~\ref{lem:aux}]
Homogenizing the equations by an additional variable $z_0$, we deduce that there is a solution in $\PP^n(k)$.  The degree assumptions directly imply that there is no solution in the hyperplane $z_0=0$, so the claim follows.
\end{proof}

To continue the proof of Lemma~\ref{lem:unique}, we spell out the coefficients of \eqref{eq:homog} and \eqref{eq:admissible} as
\[
a_i = \sum_{j=0}^i \alpha_{i,j}t^i,\quad b_i =   \sum_{j=0}^i \beta_{i,j}t^i
\]
Converting \eqref{eq:homog} to \eqref{eq:unique} by way of the admissible transformation \eqref{eq:admissible} amounts to solving the following system of 21 equations:
\begin{align*}
0  = \alpha_{10,2j} & =   \beta_{5,j}^2, \\ 
0  = \alpha_{18,2j} & =    \beta_{9,j}^2 + \sum_{l=0}^{2j} \alpha_{9,l}\beta_{9,2j-l} + \sum_{l=0}^{4} \alpha_{10,2j-2l}\beta^2_{4,l},\\
& \hphantom{= \beta_{9,j}^2}\;
+\sum_{l=0}^4 \alpha_{14,2j-l}\beta_{4,l}
 \quad (j=0,\hdots,9)\\
 0  = \alpha_{18,4j+1} & = \beta_{4,j}^4 + \sum_{l=0}^{2j} \alpha_{10,4j+1-2l} \beta_{4,l}^2 + \sum_{l=0}^{4j+1} \alpha_{14,l}\beta_{4,4j+1-l}\\
& \hphantom{= \beta_{4,j}^4}\;  + \sum_{l=0}^{4j+1} \alpha_{9,l} \beta_{9,4j+1-l} \quad (j=0,\hdots,4). 
\end{align*}
Considering the $\beta_{i,j}$ as variables (21 in number), we would like to apply Lemma~\ref{lem:aux}.  To this end, we square the middle equations and subtract appropriate multiples of the lower ones.  This does not alter the zero set but eliminates the highest powers of $\beta_{4,l}$ in the middle expressions.  Therefore, the conditions of Lemma~\ref{lem:aux} are satisfied by the new system of equations, and the claim follows.
\end{proof}

\begin{corollary}
\label{cor:rigid}
The set of admissible transformations \eqref{eq:admissible} converting \eqref{eq:homog} to \eqref{eq:unique} is finite.
\end{corollary}

\begin{proof}
This is implicit in the proof of Lemma~\ref{lem:unique}.  Consider the zero set $Z\subset\mathbb A^{21}$ given by the 21 equations above.  If $Z$ were positive-dimensional, then $\bar Z\subset \PP^{21}$ would intersect any hyperplane non-trivially.  However, we argued that $\bar Z\cap Z(z_0)=\emptyset$, so this would lead to a contradiction.
\end{proof}

Philosophically, one should consider \eqref{eq:unique} as a replacement of the Weierstrass equation of a smooth genus~$1$ curve with a rational point.  Indeed, we shall see soon that in the case of (quasi-elliptic) Enriques surfaces, it shares many convenient features with the standard Weierstrass form.  For instance, we will see this in action when working out explicit linear systems in Sections~\ref{ss:62} and~\ref{ss:VIII}.

The analogous equations for general nodal Enriques surfaces (without the assumption of being quasi-elliptic,\and in fact in any characteristic) are to be exploited in future work.

\section{Relative Jacobian}
\label{s:rJ}

In this section we work out the Weierstrass form of the relative Jacobian of the quasi-elliptic fibration \eqref{eq:unique}.  By Queen \cite{Q2}, the relative Jacobian of \eqref{eq:unique} is given by omitting the constant term:
\begin{eqnarray}
\label{eq:Jac}
\Jac(S): \;\; y^2 + a_9y  = tx^4  + ta_{4}^2x^2 + a_{14}x.
\end{eqnarray}

\begin{lemma}
\label{lem:jac}
The relative Jacobian admits the Weierstrass form
\begin{eqnarray}
\label{eq:Jac_WF}
\label{form2}
Y^2 = X^3 + (a_9^2t+a_4^4t^2)\,X
+ a_{14}^2t
\end{eqnarray}
\end{lemma}

\begin{proof}
To convert \eqref{eq:Jac} to Weierstrass form, we formally have to distinguish whether $a_9\not\equiv 0$ or not.

Assume that {$a_9\not\equiv 0$}.  We convert to Queen's second standard from \cite{Q1} by the change of coordinates
$$
  y = y_1 + \frac{a_{14}}{a_9}x + 
 \left(\frac{a_4^2t}{a_9} + \frac{a_{14}^2}{a_9^3}\right)x^2, 
$$
as we get
$$
  y_1^2 + a_9y_1 =  \underbrace{\left(t+\frac{a_4^4t^2}{a_9^2} + \frac{a_{14}^4}{a_9^6}\right)}_hx^4.
$$
The change of coordinates $x=x_2/y_2$, $y_1 = x_2/y_2^2$
gives the cubic
$$
  x_2 + a_9y_2^2=
  hx_2^3.
$$
Writing $y_2 = y_3/a_9^5h$, $x_2 = x_3/a_9^3h$, we derive the following equation which is monic in $x_3$ and in $y_3$:
\[
y_3^2 = x_3^3 + (a_4^4a_9^4t^2 + a_9^6t + a_{14}^4)x_3.
\]
This simplifies further by setting 
$$
x_3 = a_9^2 X + a_{14}^2, \quad  y_3 = a_9^3Y + a_{14}a_9^2X + a_{14}a_4^2a_9^2t
$$
and results exactly in the Weierstrass form \eqref{eq:Jac_WF}.

If {$a_9\equiv 0$}, then the affine chart at $\infty$ with coordinates $Y=y/x^2$, $X= 1/x$ readily returns a cubic starting from \eqref{eq:Jac}.  This is easily transformed into Weierstrass form -- and  yields exactly \eqref{form2} with $a_9\equiv 0$.
\end{proof}

\section{Rationality vs minimality}

By \cite{BM}, if $S$ is an Enriques surface, then the relative Jacobian $\Jac(S)$ is a rational surface.  This property will be analysed using the discriminant of \eqref{eq:Jac_WF},
 \[
\Delta =  (a_9^2t+a_4^4t^2)a_9^4+a_{14}^4,
\]
a homogeneous polynomial of degree $56$.  It follows from \cite[Proposition~5.5.3]{CD} that $\Delta$ is unique up to multiplication by $\supth{12}$ powers; this ambiguity is related to non-minimal Weierstrass forms and will play an important role in the proof of the following result.
 
\begin{proposition}
\label{prop:rat-div}
The relative Jacobian of a quasi-elliptic surface with normal form \eqref{eq:unique} is rational if and only if there are polynomials $g$ and $a_i$ $($of degree $i\,)$ such that \eqref{eq:unique} reads
 \begin{eqnarray}
\label{eq:model}
y^2 + a_1 g^2 y = tx^4 + a_0tg^2 x^2 + a_2 g^3 x + t^3 a_3^4
\end{eqnarray}
and $\Delta\not\equiv 0$.
 \end{proposition}
 
 \begin{proof}
The main subtlety is that the Weierstrass form \eqref{form2} may (and has to) be non-minimal at certain places of $\PP^1$.  Indeed, after minimalizing, rationality requires by \cite{I} that the resulting discriminant has degree~$8$ as a homogeneous polynomial, \textit{i.e.}~including contributions at $\infty$.  Thus the Weierstrass form \eqref{eq:Jac_WF} is highly non-minimal; in terms of the discriminant, there is a degree $4$ polynomial $g$ such that
\[
g^{12}\mid \Delta.
\]
As one of the special features of characteristic $2$, we have the same divisibility property for the formal derivative:
\[
g^{12} \mid \Delta' = a_9^6 \quad \Longrightarrow \quad g^{2} \mid a_9=g^{2}a_1.
\]
In turn, the shape of $\Delta$ then implies that $g^{2}\mid a_{14}= g^{2} a_6$, and moreover
\[
g^{4} \mid a_4^4t^2a_1^4+a_6^4 \quad \Longrightarrow \quad g^{2} \mid ta_4^2a_1^2+a_6^2.
\]
Since the last sum decomposes into an even and an odd part, we deduce as before, using the formal derivative, that
\begin{eqnarray}
\label{eq:div}
g\mid a_4a_1 \quad \text{and} \quad g\mid a_6 \quad (\text{so}\ g^{3}\mid a_{14}).
\end{eqnarray}
In view of the degrees of the polynomials involved, these divisibility properties are quite restrictive, especially the left-most one.  We will make use of this to prove the following important simplification. 

\begin{lemma}
\label{lem:div}
In the above setting, we have $g\mid a_4$ $($in addition to $g^2\mid a_9$, $g^3\mid a_{14})$.
\end{lemma}

\begin{proof}[Proof of Lemma~\ref{lem:div}]
Assuming the contrary, there is a linear form $\ell$ dividing $g$ with multiplicity $m$ such that $\ell^m\nmid a_4$.  Then \eqref{eq:div} implies that $\ell^{2m+1}\mid a_9$.  By the universality of the normal form \eqref{eq:unique} (\textit{cf.} Remark~\ref{rem:choice}\eqref{r:c-2}), we may as well assume that $\ell=t$.  Thus the Weierstrass form of the relative Jacobian reads
\begin{eqnarray}
\label{eq:WF-proof}
Y^2 = X^3 + (a_9^2t+a_4^4t^2)\,X
+ a_{14}^2t. 
\end{eqnarray}
We first assume that $a_9\not\equiv 0$. Then \eqref{eq:div} predicts the precise vanishing orders of $a_4$ and $a_9$ for degree reasons and by assumption:
\[
a_9 = t^{2m+1}b_9, \quad a_4 = t^{m-1}b_4, \quad \text{where}\ t\nmid b_4b_9.
\]
(Here, for ease of notation, the indices of the $b_i$ cease to indicate the degree, but rather refer back to the coefficients $a_i$ which the $b_i$ originate from.)  Then the Weierstrass form \eqref{eq:WF-proof} can be minimalized $m-1$ times by setting $X = t^{2m-2} X'$, $Y=t^{3m-3}Y'$, but the resulting Weierstrass form
\begin{eqnarray}
\label{eq:this}
Y'^2 = X'^3 + (t^7b_9^2 + t^2 b_4^4) X' + t^7 b_{14}^2
\end{eqnarray}
is minimal since the special fibre has type I$_n^*$ for some $n>0$ by \cite[Table 4.3]{CDL}.  In fact, since \eqref{eq:this} has discriminant still divisible by $t^{12}$ by construction, we infer $n\geq 8$, but then the contribution of this fibre to the Euler--Poincar\'e characteristic already prevents $\Jac(S)$ from being rational.

For $a_9\equiv 0$, the argument is completely analogous, as we have to minimalize at most $m-1$ times to arrive at the same kind of fibre type, so we skip the details.
\end{proof}

To complete the proof of Proposition~\ref{prop:rat-div}, we show that, with the divisibilities of Lemma~\ref{lem:div} in effect, the relative Jacobian $\Jac(S)$ is indeed verified to be rational generally.  To this end, we shall minimalize \eqref{form2} at each zero of $g$ counted with multiplicity.  Indeed, in terms of the factorizations
\[
a_9 = g^2 a_1, \quad a_4 = \sqrt{a_0} g, \quad a_{14} = g^3a_2
\]
implied by Lemma~\ref{lem:div}, the minimal model is given by
\begin{eqnarray}
\label{eq:Jac-minimal}
\Jac(S): \;\; 
Y^2 = X^3 + (a_1^2t+a_0^2t^2)\,X
+ a_{2}^2t.
\end{eqnarray}
Generally, this defines a rational surface by \cite[Theorem~1.1]{I} and has discriminant
\[
  \Delta (t) = (a_1^2t + a_0^2t^2)a_1^4 + a_2^4
\]
of degree $8$, as required.
\end{proof}

For later use, we record the possible configurations of reducible fibres of $\Jac(S)$.

\begin{lemma}
\label{lem:config}
If $\Delta\neq 0$, then the reducible fibres of\, $\Jac(S)$ are determined as follows:
$$
\renewcommand{\arraystretch}{1.2}
\begin{array}{c|c|c|l}
\# \text{ Roots of } \Delta & \text{Multiplicities} & \text{Fibre types} & \text{~~~~~~Conditions} \\
\hline 
8 & 1^8 & 8 \times {\III} & a_1\not\equiv 0, a_1\nmid a_2\\
\hline
5 & 1^4, 4 & 4 \times \III + \I_0^* & a_1\not\equiv 0, a_1\mid a_2, 
a_1^4\mid \Delta \\
3 & 1^2, 6 & 2 \times \III + \I_2^* 
& a_1\not\equiv 0, a_1\mid a_2, 
a_1^6\mid \Delta \\
2 & 1, 7 & \III + \III^* &  a_1\not\equiv 0, a_1\mid a_2, 
a_1^7\mid \Delta \\
\hline
2 & 4^2 & 2 \times \I_0^* & a_1\equiv 0, a_2\not\in k[t]^2\\
1 & 8 & \I_4^* & a_1\equiv 0, a_2\in k[t]^2, a_0\neq 0\\
1 & 8 & \II^* & a_1\equiv 0, a_2\in k[t]^2, a_0= 0
\end{array}
$$
\end{lemma}

\begin{proof}
This follows  from Ito; \textit{cf.} \cite[Proposition~4.2]{I}.
\end{proof}

\section{Singularity analysis I}
\label{s:ade}

We turn to the model of $S$ which we have derived in \eqref{eq:model}.  We start by analysing the singularities outside the fibres above the zeroes of $g$.  For simplicity, we first consider the fibre at $t=0$.  The notation of ADE singularities follows \cite{Artin}.

\begin{proposition}
\label{prop:ade}
Let $g(0)\neq 0$.  Then \eqref{eq:model} has at worst ADE singularities in the fibre above $t=0$, and the surface $S$ has a simple fibre of the following type at $t=0$:
\begin{table}[ht!]
  $$
  \renewcommand{\arraystretch}{1.1}
\begin{array}{c|c|c}
\text{Conditions} & \text{ADE configuration} & \text{Fibre type}\\
\hline
t\nmid a_1 \; \text{ or } \; t\nmid a_2 & - & \II \; \text{ or } \; \III\\
t\mid a_1,\; t\mid a_2, \; t^2\nmid a_2 & 4 \times A_1 & \I_0^*\\
a_0\neq 0, \; t\mid a_1\neq 0,\; t^2\mid a_2 & 2 \times A_3 & \I_2^*\\
a_0\neq 0,\; a_1 = 0,\; t^2\mid a_2 \neq 0 & 2 \times D_4^0 & \I_4^*\\
a_0=0,\; t\mid a_1\neq 0,\; t^2\mid a_2 & A_7 & \III^*\\
a_0 = a_1 = 0, \; t^2\mid a_2\neq 0 & D_8^0 & \II^*
\end{array} 
$$
\caption{ADE configurations and fibre types at $t=0$}
\label{tab:ade}
\end{table}
\end{proposition}

\begin{proof}
Starting from the affine chart \eqref{eq:model}, we first check the point at infinity -- the cusp if the fibre is irreducible.  Here the partial derivative with respect to $t$ always returns $1$, so this point is never a surface singularity.

It thus suffices to analyse the affine chart \eqref{eq:model}.  By the Jacobi criterion, the fibre at $t=0$ contains a singularity if and only if $t\mid a_1, a_2$; this verifies the first entry of Table~\ref{tab:ade}.

Assume that $t\mid a_1, a_2$, so we can write $a_1=tc_0$, $a_2=tb_1$.  In this case, Lemma~\ref{lem:config} predicts that the fibre of $S$ is non-reduced, and we will check how the double fibre component $\Theta=\{t=y=0\}$ fits into this fibre.

Recall that $g(0)\neq 0$ by assumption. For ease of notation, we shall assume that $g(0)=1$ by absorbing the non-zero factor $g(0)$ into the $a_i$.  Note that this does not affect the divisibility conditions in Table~\ref{tab:ade}.

The fibre arises by resolving the singularities at the points $(x,y,t)=(\alpha,0,0)$, where $\alpha$ runs through the roots of the auxiliary polynomial
$$
r=x^4+a_0x^2+b_1(0)x,
$$
again by the Jacobi criterion.  We now analyse the resolution of the singularities.  To this end, we blow up along~$\Theta$; in the affine chart $y=ty'$, we obtain the strict transform
\begin{eqnarray}
\label{eq:eqq}
ty'^2 + c_0tg^2y' = x^4 + a_0g^2x^2 + b_1g^3x+t^2a_3^4.
\end{eqnarray}
The exceptional curves of this blow-up, serving as fibre components at $t=0$, are encoded in the zeroes of the auxiliary polynomial $r$.  In particular, if $t\nmid b_1$, then there are four disjoint components, forming a simple $\I_0^*$ configuration together with the strict transform $\tilde\Theta$ of $\Theta$.  Since \eqref{eq:eqq} is smooth at $t=0$, as we see by inspection of the partial derivative with respect to $x$, and the same holds true for the other affine chart of the blow-up, we confirm the second entry of Table~\ref{tab:ade}.  We note that \eqref{eq:model} describes an $A_1$ singularity at each of the four points.

Assume that $b_1=tb_0$ and $a_0\neq 0$ so that $r$ has two double zeroes; in \eqref{eq:eqq} these result in disjoint double fibre components $\Theta'$, $\Theta''$ at $t=0$.  In case $c_0\neq 0$, each component contains two $A_1$ singularities at $y'=0$ and at $y'=c_0$.  Altogether, this results in a simple $\I_2^*$ configuration.  Indeed, \eqref{eq:model} is easily converted to the normal form of an $A_3$ singularity at the two singular points corresponding to the zeroes of $r$. This confirms the third entry of Table~\ref{tab:ade}.

On the other hand, if $b_1=tb_0$ and $a_0\neq 0$ but $c_0=0$, then each $\Theta'$ and $\Theta''$ contains a single singular point given by $y'=0$.  One verifies that the singularity on each, $\Theta'$ and $\Theta''$, has type $A_3$, and the resulting simple configuration corresponds to Kodaira type $\I_4^*$.  In fact, \eqref{eq:model} readily displays a singularity of type $D_4^0$ at $(t,x,y) = (0,0,0)$ in the normal form from \cite{Artin}, and similarly at the other double root of $r$.  This proves the fourth entry of Table~\ref{tab:ade}.

It remains to cover the case $b_1=tb_0$ and $a_0=0$.  Then \eqref{eq:eqq} reveals the 4-fold fibre component $\Theta'=\{t=x=0\}$.  We continue to blow up along $\Theta'$.  In the affine chart $t=xt'$, the strict transform reads
\begin{eqnarray}
\label{eq:eqqq}
t'y'^2 + c_0t'g^2y = x^3+t'b_0g^3x+t'^2a_3^4x.
\end{eqnarray}
At $x=0$ (which describes part of the fibre at $t=xt'=0$), we obtain two simple fibre components given by $y'=0$ and $y'=c_0$ if $c_0\neq 0$.  Thus the fibre itself is simple, as claimed, and two analogous further blow-ups add another two components each, of multiplicity $2$, resp.\ $3$, to make for the fibre of Kodaira type III$^*$.  Accordingly, in \eqref{eq:model}, the lowest-order terms $y^2, ty, tx^4$ encode a singularity of type $A_7$, as claimed.

Meanwhile, if $c_0=0$, then \eqref{eq:eqqq} returns another double fibre component given by $\Theta''=\{x=y'=0\}$.  It contains an $A_5$ singularity at $(0,0,0)$, but more importantly, there is also an $A_1$ singularity in the other chart $x=tx'$.  Its resolution results in a simple fibre component, confirming the claim of the lemma.  The overall configuration of exceptional curves gives a fibre of Kodaira type II$^*$.  In \eqref{eq:model}, the lowest-order terms $y^2, tx^4, t^2x, t^3$ confirm a singularity of type $D_8^0$ in the notation of \cite{Artin}.  This completes the proof of Proposition~\ref{prop:ade}.
\end{proof}

With a view towards our goal of determining when $S$ is an Enriques surface (so that it has exactly one or two multiple fibres), we record the following useful consequence. 

\begin{corollary}
\label{lem:ade}
\label{lem:simple}
Let $t_0\in\PP^1$ be such that $g(t_0)\neq 0$.  Then \eqref{eq:model} has at worst ADE singularities in the fibre above $t_0$, and the surface $S$ has a simple fibre at $t_0$.
\end{corollary}

\begin{proof}
In order to reduce the corollary to Proposition~\ref{prop:ade}, we first apply a M\"obius transformation to move the special fibre to $t=0$.  Then a suitable variable transformation
\[
(x,y) \longmapsto (ux+b_4, vy+\alpha x^2 + \beta g x + b_9), \quad u,v,\alpha, \beta\in k, \;\; b_4, b_9\in k[t]
\]
ensures that the special shape of \eqref{eq:model} with the given divisibilities by $g$ is restored.  This can be viewed as a special instance of the universality of \eqref{eq:unique}, as explained in Remark~\ref{rem:choice}\eqref{r:c-1} and extended to \eqref{eq:model}.  The claim of Corollary~\ref{lem:ade} follows directly.
\end{proof}

\section{Singularity analysis II}
\label{s:res}

Let $\alpha$ be a root of $g$.  In analogy with the case of (quasi-)elliptic surfaces with section, we call $\alpha$ \emph{non-minimal} (for the model \eqref{eq:model}) if $a_3(\alpha)=0$.  Indeed, at a non-minimal root, we can apply a change of variables
\[
x= (t-\alpha)x', \quad y=(t-\alpha)^2y',
\]
reducing the degree of \eqref{eq:model} by $4$ to $14$ (and those of all coefficients accordingly) while embedding the surface in $\PP[1,1,3,7]$.  In analogy, the degree of the discriminant drops by $12$ -- exactly as in the proof of Lemma~\ref{lem:div} (or in the last step of Tate's algorithm; \textit{cf.} \cite{Tat}).

\begin{proposition}
\label{prop:sing}
At the roots of $g$, either \eqref{eq:model} is non-minimal, or $S$ has a double fibre.
\end{proposition}

Since $S$ is an Enriques surface, the proposition together with Corollary~\ref{lem:ade} has the following important consequence (disregarding multiplicities). 

\begin{corollary}
\label{cor:number}
The polynomial $g$ has exactly two minimal roots if $S$ is classical and one minimal root if $S$ is supersingular.
\end{corollary}

\begin{proof}[Proof of Proposition~\ref{prop:sing}]
Let $\alpha$ be a minimal root of $g$ (of multiplicity $m\geq 1$).  By the universality of \eqref{eq:unique}, we may assume that $\alpha=0$ and write
$$
a_4 = t^{m} b_4, \quad a_9 = t^{2m} b_9, \quad a_{14} = t^{3m} b_{14}
$$
as before.  We proceed by resolving the singularity at $(0,0,0)$ until we reach a model with ADE singularities at worst.  The first two steps where we blow up along smooth curves rather than in single points are exactly as in the proof of Proposition~\ref{prop:ade} (but with the additional divisibilities of coefficients provided by the root of $g$).  Indeed, the fibre of \eqref{eq:model} at $t=0$ is the double component $\Theta_0=\{t=y=0\}$, and we set out by blowing up $S$ along $\Theta_0$.  It suffices to consider the following chart. 

\subsection{First blow-up: \fbox{$y=ty'$}}
\label{ss:1st}

Then the strict transform of \eqref{eq:model} is
\begin{eqnarray}
\label{eq1}
ty'^2 + t^{2m}b_9 y'  = x^4 +  t^{2m}b_4^2x^2 + t^{3m-1}b_{14}x+t^2a_3^4
\end{eqnarray}
with 4-fold fibre component $\Theta_1=\{t=x=0\}$ and singular point at $(0,0,0)$.  We continue to blow up along $\Theta_1$. Again, one affine chart suffices to investigate the exceptional curve.

\subsection{Second blow-up: \fbox{$t=xt'$}} 
\label{ss:2nd}

The strict transform of \eqref{eq1} reads
\label{blu2}
\begin{align}
\label{eq2}
t'y'^2 + t'^{2m} b_9(xt') x^{2m-1} y' & =   x^3 +t'^{2m} b_4(xt')^2x^{2m+1}\nonumber\\
& \hphantom{=}\, + t'^{3m-1}b_{14}(xt')x^{3m-1}+t'^2a_3(xt')^4x.
\end{align}
At $x=0$, we recover the 4-fold fibre component $\Theta_1=\{x=t'=0\}$ as well as the double fibre component $\Theta_2= \{x=y'=0\}$, intersecting $\Theta_1$ transversely in the surface singularity $(0,0,0)$.  By inspection of the threefold vanishing order of each monomial, we see that this is not an ADE singularity, so we blow it up in the usual manner.  Again, one affine chart suffices to detect the remaining singularities. 

\subsection{Third blow-up: \fbox{$x=t'x'', y'=t'y''$}}
\label{ss:3rd}

The strict transform of \eqref{eq2} is given by
\begin{align}
y''^2  + t'^{4m-3} b_9(x''t'^2) x''^{2m-1} y''& = 
x''^3 +t'^{4m-2} b_4(x''t'^2)^2x''^{2m+1}\nonumber\\
\label{eq3}
& \hphantom{=}\, +t'^{6m-5}b_{14}(x''t'^2)x''^{3m-1}+a_3(x''t'^2)^4x''.
\end{align}
Since $a_3(0)\neq 0$ by our minimality assumption on $\alpha$, we can rescale to assume that $a_3(0)=1$.  At $t'=0$, we thus obtain the cuspidal cubic
$$
\Theta_3=\{t'=y''^2+x''^3+x''=0\}
$$
from \eqref{eq3}.  Since $t=x''t'^2$, this has multiplicity $2$ as a fibre component (and this multiplicity will persist throughout the resolution process to give the claim of the proposition).  Note that $\Theta_3$ is Cartier; hence the surface is smooth away from the cusp.

\subsection{Interlude: Contraction of fibre components}
\label{ss:interlude}

For completeness, we briefly deviate from the resolution of singularities to explain how the fibre components $\Theta_0,\Theta_1,\Theta_2$ behave.  Recall that the fibre $\hat F$ above $t=0$ satisfies
\[
\hat F = 2 \Theta_0 + 4 \Theta_1 + 2\Theta_2 + 2\Theta_3.
\]
Here the components $\Theta_1$ and $\Theta_2$ intersect $\Theta_3$ transversely in two distinct points (different from the cusp).  On the partial resolution $\hat S$ given by \eqref{eq3}, where we have blown up $S$ three times, we thus obtain the following configuration of curves making for the support of $\hat F$:

\begin{figure}[ht!]
\setlength{\unitlength}{.45in}
\begin{picture}(7,3)(0,0)
\thinlines

\qbezier(2,0)(3.5,.5)(3.5,3)
\qbezier(5,0)(3.5,.5)(3.5,3)
\put(3.6,2.4){$\Theta_3$}

\put(0,2.4){$\Theta_0$}
\put(0.5,0){\line(0,1){3}}
\put(0,0.5){\line(1,0){3}}
\put(4,0.5){\line(1,0){3}}

\put(1.5,.65){$\Theta_1$}
\put(6.5,.65){$\Theta_2$}

\end{picture}
\end{figure}

By the moving lemma, one has $\Theta_i\cdot\hat F=0$ for each $i=0,1,2$, so the multiplicities calculated above give
\[
\Theta_0^2=-2, \quad \Theta_1^2=\Theta_2^2 = -1.
\]
Hence we can contract first $\Theta_1, \Theta_2$ and then $\Theta_0$ to two smooth points on a relatively minimal, but possibly still singular, model $S'$ as sketched below.  This surface shares the same affine equation \eqref{eq3} as $\hat S$, so we will continue to argue with the double fibre $\Theta_3$ of $S'$ given by \eqref{eq3} (and iterate this procedure if necessary).

$$
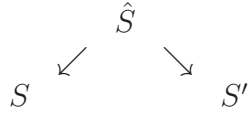
\begin{array}{ccccc}
&& \hat S&&\\
&\swarrow&&\searrow&\\
S &&&& S'
\end{array}
$$

\subsection{Singularity analysis continued}

In order to analyse the singularity at the cusp $(1,0)$ (if any), we introduce a new variable
$$
\hat x = 
x''+a_3^2(x''t'^2) 
$$
so that the cusp is located at $(\hat x, y'', t') = (0,0,0)$.  The analysis is facilitated by passing to the completion $k[[\hat x,t']]$. We define the unit $b_3\in k[[\hat x,t']]$ implicitly by the condition $b_3(\hat x,t') = a_3(x''t'^2)$, so that
\begin{eqnarray}
\label{eq:hatx}
x'' = \hat x + b_3^2.
\end{eqnarray}
When we  subsequently also write $y''=\hat y + b_3 \hat x + t'^{2m-1}b_3b_4\,(\hat x+b_3^2)^m$ but suppress the arguments of $b_3$ and $b_4$ for ease of presentation, \eqref{eq3} transforms as
\begin{align}
\hat y^2  + t'^{4m-3}  b_9
\, (\hat x + b_3^2)^{2m-1}
 (\hat y + b_3 \hat x + t'^{2m-1}b_3b_4\,(\hat x+b_3^2)^m)
 \nonumber\\
 =  \hat x^3 
+t'^{4m-2} b_4^2\,(\hat x+b_3^2)^{2m}\hat x
\label{eq4}
+ t'^{6m-5}b_{14}\,(\hat x+b_3^2)^{3m-1}, 
\end{align}
where the argument of each of $b_4, b_9, b_{14}$ is $t'^2(\hat x + b_3^2)$.  Looking closely, the lowest-order terms exactly resemble the equation of an elliptic curve in Weierstrass form, but it can be minimalized $m-1$ times.  Hence, without considering any further intermediate partial resolutions, we apply the variable change
\begin{eqnarray}
\label{eq:tilde}
(t',\hat x,\hat y) = (\tilde t, \tilde t^{2m-2} \tilde x, \tilde t^{3m-3}\tilde y), 
\end{eqnarray}
which leads to the following strict transform of \eqref{eq4}:
\begin{align}
\tilde y^2  + \tilde t  b_9
\, (\tilde t^{2m-2} \tilde x + b_3^2)^{2m-1}
 (\tilde t^{m-1}\tilde y + b_3 \tilde x + \tilde t b_3b_4\, (\tilde t^{2m-2} \tilde x+b_3^2)^m)
 \nonumber\\
 =  \tilde x^3 
+\tilde t^{2} b_4^2\,(\tilde t^{2m-2} \tilde x+b_3^2)^{2m}\tilde x
\label{eq5}
+ \tilde tb_{14}\,(\tilde t^{2m-2} \tilde x+b_3^2)^{3m-1}.
\end{align}
At $\tilde t=0$, this again describes a cuspidal cubic, say $\Theta'=\{\tilde t=\tilde y^2 + \tilde x^3 = 0\}$.  Since $t=\tilde t^2(\tilde t^{2m-2} \tilde x + b_3^2)$ with unit $b_3$, the cuspidal cubic has multiplicity $2$ as a fibre of $S$.  On $\Theta'$, a singularity of this surface may exist only at the cusp point if at all.  Indeed, the singularity is necessarily rational,  and the types are given as follows (this can also be read off from the relative Jacobian):
\begin{enumerate}
\item If $t \nmid b_{14}$, then non-singular.
\item If $t\mid b_{14}$, $t \nmid b_{9}$, then $A_1$.
\item If $t\mid\mid b_{14}$, $t\mid b_{9}$, then $D_4$.
\item If $t^2\mid b_{14}$, $t\mid\mid b_9$ and $t\nmid b_4$,  then $D_6$.
\item\label{case5} If $t^2\mid b_{14}$, $t\mid\mid b_9$ and $t\mid b_4$, then   $E_7$.
\item\label{case6} If $t^2\mid\mid b_{14}$,  $t^2\mid b_9$ and $t\nmid b_4$, then $D_8$.
\item\label{case7} If $t^2\mid\mid b_{14}$, $t^2\mid b_9$ and $t\mid b_4$, then $E_8$.
\end{enumerate}
In particular, the multiplicity of the fibre continues to be $2$ since the strict transform of $\Theta'$ (of multiplicity $2$) presents a simple component of the underlying Kodaira type.  This completes the proof of Proposition~\ref{prop:sing}.
\end{proof}

\begin{remark}
Note that the above computations are of purely local nature.  Hence, if $S$ is an Enriques surface, then the divisibility properties from Lemma~\ref{lem:div} limit the possible shapes of $b_4$ and $b_9$.  In practice, this means that $b_4\equiv 0$ in cases~\eqref{case5} and~\eqref{case7} and $b_9\equiv 0$ in cases~\eqref{case6}  and~\eqref{case7}.
\end{remark}

\section{Canonical divisor}
\label{s:canon}

In order to decide when \eqref{eq:model} (or \eqref{eq:unique}) defines an Enriques surface, we now investigate the canonical divisor.  Since the resolution of ADE singularities does not affect it and non-minimal roots of $g$ can be dealt with by lowering the degree, this amounts to analysing the minimal roots of $g$ (which are one or two in number by Corollary~\ref{cor:number}).  Throughout, we argue with the standard rational 2-form
\[
\omega = dx\wedge dt/a_9 \quad \text{if}\ a_9\not\equiv 0, \quad \text{or}\quad  \omega = dy\wedge dt/a_{14} \quad \text{if}\ a_{14}\not\equiv 0.
\]
(Note that $(a_9,a_{14})\not\equiv (0,0)$ since otherwise \eqref{eq:unique} and \eqref{eq:model} 
would be geometrically reducible.)

\begin{proposition}
\label{prop:can}
Let $\alpha\neq \infty$ denote a minimal root of $g$ of multiplicity $m$.  Then $\omega$ extends over a minimal resolution of \eqref{eq:model} with a pole of order $m/2$ along the fibre at $\alpha$.
\end{proposition}

\begin{remark}
At $\infty$ we have to be more careful since our choice of $\omega$ tends to have a zero there (as we shall discuss below around Proposition~\ref{prop:multi}).  Of course, the assumption $\alpha\neq \infty$ in Proposition~\ref{prop:can} can always be achieved by some M\"obius transformation, so it should not be viewed as a restriction.
\end{remark}

\begin{proof}[Proof of Proposition~\ref{prop:can}]
As before, we may assume that $\alpha=0$.  Then we simply trace back $\omega$ through the resolution of the singularity in the proof of Proposition~\ref{prop:sing}.  For brevity, we only discuss the case $a_9\not\equiv 0$. The other case is completely analogous.  Step by step, we obtain
\begin{align*}
\omega & \hphantom{(}=\hphantom{)}   \frac{dx\wedge dt}{a_9} = \frac{dx\wedge dt}{t^{2m}b_9} \stackrel{\ref{ss:2nd}}{=}\frac{dx\wedge dt'}{x^{2m-1}t'^{2m}b_9}
\stackrel{\ref{ss:3rd}}{=} \frac{dx''\wedge dt'}{x''^{2m-1}t'^{4m-2}b_9}\\
& 
\stackrel{\eqref{eq:hatx}}{=}
\frac{d\hat x\wedge dt'}{(\hat x+b_3^2)^{2m-1}t'^{4m-2}b_9}
\stackrel{\eqref{eq:tilde}}{=} \frac{d\tilde x \wedge d\tilde t}{(\tilde t^{2m-2}\hat x+b_3^2)^{2m-1} \tilde t^{2m} b_9}. 
\end{align*}
Since \eqref{eq5} has at worst ADE singularities in the fibre at $\tilde t=0$ by the proof of Proposition~\ref{prop:sing}, there is the standard 2-form
\[
\tilde \omega = \frac{d\tilde x \wedge d\tilde t}{(\tilde t^{2m-2} \tilde x + b_3^2)^{2m-1}
 \tilde t^{m}b_9}.
 \]
This extends to a 2-form on the minimal desingularization which is regular and non-zero along the fibre.  Comparing $\omega$ and $\tilde\omega$, we deduce that $\omega$ has a pole given by $\tilde t^{m}$. In particular, this gives the claimed order.
\end{proof}

{\samepage
\begin{proposition}
\label{prop:multi}\leavevmode
\begin{enumerate}
\item
For $S$ to be a classical Enriques surface, $g$ has to have exactly two minimal roots, each of multiplicity $1$.
\item
For $S$ to be a supersingular Enriques surface, $g$ has to have exactly one minimal root, of multiplicity $2$.
\end{enumerate}
\end{proposition}
}

\begin{proof}
By Corollary~\ref{cor:number}, $g$ has exactly the claimed number of minimal roots.  Denote their multiplicities by $m_1, m_2$, and set $d=m_1+m_2$.  Since being Enriques is a birational property, we may eliminate all non-minimal roots of $g$ in \eqref{eq:model}.  Since there are $4-d$ non-minimal roots counted with multiplicities, this leads to the following model of degree $4d+2$ in $\PP[1,1,d,2d+1]$:
 \begin{eqnarray}
\label{eq:model-d}
y^2 + a_1 g_d^2 y = tx^4 + a_0tg_d^2 x^2 + a_2 g_d^3 x + t^3 c_{d-1}^4.
\end{eqnarray}
The standard rational $2$-forms $\omega = dx\wedge dt/a_1g_d^2$ or $dy\wedge dt/a_2g_d^3$ are regular outside the fibres above the roots of $g_d$, attaining a zero of order $d-1$ at $\infty$ (verified by considering the third chart, with $s=1/t$, from Section~\ref{ss:charts}).  By Proposition~\ref{prop:can}, $\omega$ extends over the minimal desingularization with
\[
\mbox{div}(\omega) = (d-1)F_\infty-\frac{m_1}2 F_1 - \frac{m_2}2 F_2.
\]
Since $m_1+m_2=d$, this is numerically trivial if and only if $d=m_1+m_2=2$, as claimed.
\end{proof}

\section{Proof of Theorem~\ref{thm}}
\label{s:pf}

\subsection{Uniform normal form}

With a view towards moduli and our applications, we first consider the following normal form of quasi-elliptic Enriques surfaces where the multiple fibres are not yet fixed at $t=0$ and $\infty$ (in the classical case). 

\begin{theorem}
\label{rem:both}
\label{thm:general}
Let $S$ be a quasi-elliptic Enriques surface. Then $S$ is given by an equation
\begin{eqnarray}
\label{eq:general}
S: \;\; y^2 + g_2^2 a_1 y = tx^4 + t g_2^2 a_0 x^2 + g_2^3 a_2 x + t^3 c_1^4.
\end{eqnarray}
Here the coefficients are polynomials of degree at most the index with the conditions that $c_1, g_2\not\equiv 0$, $(a_1,a_2)\not\equiv (0, 0)$ and $c_1\nmid g_2$ $($and $\deg(g_2)=2$ if $\deg(c_1)=0)$.

The Enriques surface is classical if $g_2$ has two different roots $($possibly including $\infty)$ and supersingular if $g_2$ is a square.
\end{theorem}

One of the benefits of the theorem is that the codimension $1$ condition for being supersingular inside the closure of the moduli space of classical Enriques surfaces becomes transparent in a very instructive and explicit way.  We will exploit this extensively in applications studied in the remainder of this paper.

\begin{proof}
By Proposition~\ref{prop:multi}, the polynomial $g$ of degree $4$ that has been central to many of our considerations has two non-minimal roots counted with multiplicity.  Hence \eqref{eq:model} can be minimalized at these two roots, and we obtain exactly the normal form \eqref{eq:general}, where $g_2$ has either two simple roots (classical case) or one double root (supersingular case).  The non-vanishing and non-divisibility conditions follow directly from what we have seen before.

The resulting surface $S$ is indeed an Enriques surface because the canonical divisor is trivial in the case $m_1=2$ and numerically trivial in the case $m_1=m_2=1$ as $K_S = F_\infty-F_1/2-F_2/2$.  In either case, the Euler--Poincar\'e characteristic equals that of $\Jac(S)$, \textit{i.e.}~$e(S) = 12$.  Together this identifies $S$ as an Enriques surface by virtue of the Enriques--Kodaira classification; \textit{cf.} \cite{BM}.
\end{proof}

\subsection{Proof of Theorem~\ref{thm}\eqref{thm-1}}

In the classical case, we apply M\"obius transformations and rescale $x,y,t$ to normalize $g_2=t$ and $c_1=1+t$ to obtain \eqref{eq:classical} from \eqref{eq:general}.  The conditions given are all immediate.  \qed

\subsection{Proof of Theorem~\ref{thm}\eqref{thm-2}}

In the supersingular case, we normalize $g_2=t^2$ and $c_1=1$ to obtain \eqref{eq:supersingular} from \eqref{eq:general}.  The conditions are again immediate.  \qed

\begin{remark}
In the supersingular case, there is one normalization left preserving the shape of \eqref{eq:supersingular}.  It amounts to the scalings $(x,y,t)\mapsto (\alpha x, \alpha^3 y, \alpha^2 t)$ for $\alpha\in k^\times$, to comply with the fact that the supersingular locus has codimension $1$ inside the closure of the classical locus.
\end{remark}

\begin{remark}
The normal forms in Theorems~\ref{thm} and~\ref{rem:both} are very well suited for explicit computations, similarly to the Weierstrass form of an elliptic curve.  We will see this in action when computing automorphism groups in Section~\ref{s:aut} and linear systems in Sections~\ref{ss:62} and~\ref{ss:VIII}.
\end{remark}

\begin{remark}
\label{rem:appl}
The techniques of this paper also have  applications beyond Enriques surfaces.  For instance, one can use them to construct explicit simply connected projective surfaces with $p_g=1$ which provide counterexamples to the Torelli theorem (\textit{cf.} \cite{Chakiris, PP}).
\end{remark}

\section{Isomorphisms of quasi-elliptic Enriques surfaces}

For later use, both for automorphisms and for moduli dimensions, we discuss the implications of our normal form for isomorphisms of quasi-elliptic Enriques surfaces.  For this purpose, we fix the generic fibres $E, E'$ of two quasi-elliptic Enriques surfaces $S, S'$, given by \eqref{eq:general}, resp., in dashed notation,
\begin{eqnarray}
\label{eq:S'}
S: \;\; y'^2 + {g_2'}^2 a'_1 y = tx^4 + t {g'_2}^2 a'_0 x^2 + {g'_2}^3 a'_2 x + t^3 {c'_1}^4.
\end{eqnarray}

\begin{lemma}
\label{lem:iso}
In the above set-up, let $g\colon E \to E'$ be an isomorphism.  Then $ g$ takes the polynomial shape
\[
g\colon (x,y) \longmapsto (ux+d_2, u^2y+d_5)
\]
for a unit $u\in k^\times$ and polynomials $d_2, d_5\in k[t]$ of degree at most $2$, resp.\ $5$.  Moreover, $u$ is unique up to possibly multiplying by a third root of unity, and there are only finitely many choices for $d_2, d_5$.
\end{lemma}

\begin{proof}
The isomorphism $g$ defines a birational map $S\dasharrow S'$ which automatically is an isomorphism since the surfaces are smooth and the canonical bundles are nef.  Since it also preserves the fibrations, we have $g_2\sim g_2'$; \textit{i.e.}~$g_2$ and $g_2'$ are proportional.  Clearly the cusp at $\infty$ is preserved: $g(P_\infty) = P_\infty'$.  Arguing as in Section~\ref{eq:def}, we thus find
\[
g^*x'\in L(P_\infty) = \langle 1,x\rangle,\quad
g^*y'\in L(2P_\infty) = \langle 1,x,x^2,y\rangle.
\]
Hence $g$ takes the form
\begin{eqnarray}
\label{eq:g1}
g\colon (x,y) \longmapsto (d_1x+d_0, e_3y+e_2x^2+e_1x+e_0)
\end{eqnarray}
for some rational functions $d_i, e_j \in k(t)$ with $d_1e_3\neq 0$.  Substituting into \eqref{eq:S'}, we compare the coefficients of $x^4$ and $x^2$, both of which are odd in \eqref{eq:general}, to derive that $e_2=e_1=0$ and subsequently, also taking the coefficient of $y^2$ into account, $e_3=d_1^2$.  But then, from the coefficients of $y, x^2$ and $x$, we derive the proportionalities
\[
d_1^{-2}a_1' \sim a_1, \quad d_1^{-2} a_0' \sim a_0, \quad d_1^{-3}a_2' \sim a_2.
\]
Since not all $a_i$ can be zero simultaneously, we derive at least one condition on $d_1$.  But since the $a_i$ have degree smaller than the negative power of $d_1$ involved, this shows that $d_1$ is a constant, \textit{i.e.}~$d_1=u\in k^\times$.  This is uniquely determined by comparing the coefficients unless $a_1=a_0=a_1'=a_0'=0$; in the latter case, uniqueness only holds up to multiplying by a third root of unity, as stated.

It remains to study the rational functions $d_0, e_0$.  We first claim that they are polynomial.  To see this, assume to the contrary that they admit a pole of order $n_1, n_2$ at $P\in \mathbb A^1$.  Since substitution of \eqref{eq:g1} into \eqref{eq:S'} has to lead to the polynomial equation \eqref{eq:general}, the poles have to cancel out.  The highest-order terms are $e_0^2+td_0^4$.  This shows that $n_2=2n_1$, but even so, since $e_0^2$ is even and $td_0^4$ is odd, we have $v_P(e_0^2+td_0^4)\leq 1-4n_1$.  Since all other terms involved have pole order at most $2n_1<4n_1-1$ at $P$, we infer that $n_1, n_2\leq 0$. Hence $d_0, e_0\in k[t]$, as claimed.

To conclude, we have to bound the degrees of $d_0, e_0$.  But here the analogous reasoning applies to prove that $\deg(d_0)\leq 2$ and $\deg(e_0)\leq 5$.  This shows the claimed shape of $g$, and the finiteness of the possible choices for $d_0, e_0$ follows from Corollary~\ref{cor:rigid} since now \eqref{eq:g1} is a special case of \eqref{eq:admissible} with degrees adjusted to accommodate the total degree of \eqref{eq:general} being $10$ instead of $18$ for \eqref{eq:homog}.
\end{proof}

We shall now use Lemma~\ref{lem:iso} to study isomorphisms $g$ of quasi-elliptic Enriques surfaces which induce a non-trivial action on the base, say by $\pi$.  This is encoded in the following commutative diagram:
\[\begin{CD}
S @>g>> S'\\
@VfVV @VV{f'}V\\
\PP^1 @>\pi>> \PP^1\rlap{.}
\end{CD}\]

\begin{lemma}
\label{lem:iso2}
In the above set-up, $g$ assumes the weighted homogeneous form
\begin{eqnarray}
\label{eq:gg}
g\colon (x,y,s,t) \longmapsto (ux+d_2, vy+d_1x^2+d_3x+d_5, \pi(s,t))
\end{eqnarray}
with $u,v\in k^\times$ and the $d_i\in k[s,t]$ homogeneous of degree $i$.  Here $d_1, d_3$ are unique, $u,v$ are unique possibly up to multiplying by a third root of unity, and there are only finitely many choices for $d_2, d_5$.
\end{lemma}

\begin{proof}
We can factor $g$ through the $\pi$-action on the base, extended trivially to $S$, composed with an isomorphism
\[
g^\pi\colon S^\pi \lra S'
\]
from the $\pi$-twisted surface $S^\pi$ to $S'$.  Then we use the unique polynomials $d_1, d_3\in k[t]$ and $\alpha\in k^\times$ such that the automorphism
\[
g'\colon (x,y,t) \longmapsto (\alpha x,y+d_1x^2+d_3x,t)
\]
converts the equation of $S^\pi$ derived from $S$ to the normal form \eqref{eq:general}.  But then we are in the situation of Lemma~\ref{lem:iso}, so the claims of Lemma~\ref{lem:iso2} follow immediately.
\end{proof}

\begin{remark}
It follows that any automorphism of a quasi-elliptic Enriques surface assumes the form \eqref{eq:gg}, as we shall exploit in Section~\ref{ss:auts}.  We note that this agrees with the form for a special case from \cite[Section~6]{KKM}.  However, the argument in \textit{loc.\ cit.}~appears to be incomplete, at least in the classical case, as it states that an affine open of $S$ is obtained by resolving the rational singularities of \eqref{eq:general} (denoted by ${\rm Spec}(A)$ in \textit{loc.~cit.}  However, we saw in Section~\ref{ss:interlude} that at the multiple fibres, there are non-rational singularities, requiring three blow-ups which are succeeded by three contractions, and possibly some further blow-ups.

To verify the automorphism groups for types $\Gamma = \tilde{E_8}$, $\tilde{D_8}$ and $\tilde{D_4}+\tilde{D_4}$, one can now safely appeal to Lemma~\ref{lem:iso2}; alternatively, the computations of the finite automorphism groups for the families in \cite{KKM} can also be verified by
\begin{itemize}
\item
limiting the non-numerically trivial automorphisms of the surfaces based on the symmetries of the finite graph $\Gamma$ of $(-2)$-curves (\textit{cf.}~Section~\ref{s:finite}),
\item
bounding the subgroup of numerically trivial automorphisms by the general results of \cite[Corollary 7.8(1) and Theorem 7.6]{DM2} and
\item
comparing this to the automorphisms of ${\rm Spec}(A)$ exhibited explicitly in \cite{KKM} (since these automatically extend to automorphisms of the Enriques surfaces as in the proof of Lemma~\ref{lem:iso}).
\end{itemize}
\end{remark}

\section{Torsor interpretation}

We continue the paper with some interesting applications of our results and techniques.  In this section we consider the general quasi-elliptic picture, but as a motivation we first recall the usual set-up.

While the standard construction of an Enriques surface outside characteristic $2$ nowadays probably is that as a quotient of a K3 surface by a free involution, there is also another approach using the Jacobian of any genus~$1$ fibration on the Enriques surface.  This is a rational elliptic surface, and over $\mathbb C$, one can recover the Enriques surface by a suitable logarithmic transformation (\textit{cf.} \cite[Section~V.13]{BHPV}).  In essence, this depends on the two ramified fibres, but it also involves a choice of $2$-torsion points on the ramified fibres.  This implies that a given rational elliptic surface $X$ admits a $2$-dimensional family of Enriques surfaces whose Jacobian is $X$, but the family is only irreducible if $X$ has no $2$-torsion section.

In the algebraic category, there is an alternative interpretation of this construction
in terms of torsors; see \cite[Section~4.10]{CDL}.
Naturally, this also applies to the quasi-elliptic fibrations on which we are focusing in this paper,
as featured in Theorem~\ref{thm2}.

\subsection{Proof of Theorem~\ref{thm2}}
\label{ss:explicit}

Let $X$ be a rational quasi-elliptic surface with section.  By \cite[Equation~(5.1)]{I}, it can be given by the standard Weierstrass form
\[
X: \;\; y^2 = x^3 + tb_2 x + t c_2^2,
\]
where $b_2, c_2\in k[t]$ are of degree $2$ as usual.  By Lemma~\ref{lem:jac}, this is exactly the relative Jacobian \eqref{form2} adjusted to our normal form \eqref{eq:general} once we identify
\[
b_2 = a_1^2 + ta_0^2, \quad c_2 = a_2.
\]
Hence Theorem~\ref{thm:general} exhibits the torsor exactly in terms of $c_1$ and $g_2$ (under the condition that $c_1\nmid g_2$) since we can always rescale $x,y$ in \eqref{eq:general} to normalize some coefficient of $c_1$ or $g_2$. As, by Lemma~\ref{lem:iso}, there can only be finitely many other symmetries, this leads to a $4$-dimensional family of curves of arithmetic genus~$1$ over $k[t]$ as predicted by Ogg--Shafarevich theory.

To understand the moduli of the underlying Enriques surfaces, it suffices to study the symmetries of the fixed quasi-elliptic fibration.  Except for the scaling alluded to above, Lemma~\ref{lem:iso2} tells us that essentially the only continuous symmetries of \eqref{eq:general} may be induced from M\"obius transformations.  Clearly, this is impossible as soon as there are three or more reducible fibres; in this case, we derive a $4$-dimensional family of classical Enriques surfaces (where $g_2$ has distinct roots) and a $3$-dimensional irreducible subfamily of supersingular Enriques surfaces (where $g_2$ has a double root), lying as torsors above the given rational quasi-elliptic surface~$X$, thus confirming Theorem~\ref{thm2} in this case.

If there are fewer reducible fibres, then there are additional M\"obius transformations preserving the relative Jacobian; in particular, if there are at least three reducible or multiple fibres, then M\"obius transformations can be used to normalize three of them, thus again confirming Theorem~\ref{thm2}.

This only leaves the supersingular case with exactly one reducible fibre.  This will be treated in Section~\ref{ss:cont'd} following general considerations for the torsors of rational quasi-elliptic surfaces.

\subsection{Explicit torsors for Ito's normal forms}

For immediate use in the proof of Theorem~\ref{thm2} and for later reference, we apply the above methods to exhibit the torsors for all rational quasi-elliptic surfaces, following the classification of Ito in \cite[Table 1]{I}.  Throughout we keep the order of Lemma~\ref{lem:config} and let $\alpha \in k, \beta\in k^\times$.

Recall that the polynomials $c_1, g_2$ of degree $1$, resp.\ $2$, do not share a common root including $\infty$ and that the torsor defines a classical Enriques surfaces if $g_2$ has two distinct roots and a supersingular Enriques surface if $g_2$ has a double root.

\begin{table}[ht!]
$$\renewcommand{\arraystretch}{1.1}
\begin{array}{c|c|c}
\text{Fibre configuration} & \text{Ito's equation: } y^2 = & \text{Enriques torsor}\\
\hline
8 \times \III
&
x^3 + (t^3+\alpha^2 t^2 +\beta^2 t) x + t^3
&
y^2 + (t+\beta) g_2^2  y = tx^4 + \alpha t g_2^2 x^2  \\
&
&
\phantom{y^2 + (t+\beta) g^2  y = }
+ tg_2^3 x + t^3c_1^4
\\
4 \times \III + \I_0^* 
&
x^3 + (t^3+\beta^2 t^2 +t) x
&
y^2 + (t+1) g_2^2  y = tx^4 + \beta t g_2^2 x^2 + t^3c_1^4
\\
2 \times \III + \I_2^*
&
x^3 + (t^3+t) x
&
y^2 + (t+1) g_2^2  y = tx^4 + t^3c_1^4
\\
\III + \III^* 
&
x^3 + t^3 x
&
y^2 + t g_2^2 y = tx^4 + t^3c_1^4
\\
2\times \I_0^* 
&
x^3 + \alpha^2 t^2 x + t^3
&
y^2 = tx^4 + \alpha t g_2^2 x^2 + t g_2^3 x + t^3c_1^4
\\
\I_4^*
&
x^3 + t^2 x + t^5
&
y^2 = tx^4 + t g_2^2  x^2 + t^2 g_2^3 x + t^3c_1^4
\\
\II^* 
&
x^3 + t^5
&
y^2 = tx^4 + t^2 g_2^3 x + t^3c_1^4
\end{array}
$$
\caption{Torsors for rational quasi-elliptic  surfaces with section}
\label{tab2}
\end{table}

\subsection{Proof of Theorem~\ref{thm2}, continued}
\label{ss:cont'd}

To complete the proof of Theorem~\ref{thm2}, it remains to study the supersingular Enriques surfaces arising as torsors above rational quasi-elliptic surfaces with just one reducible fibre.  We distinguish two cases by the possible fibre types $\I_4^*$ and $\II^*$ at $t=0$ as in Table~\ref{tab2} and locate the multiple fibre at $\infty$; after normalizing, this amounts to setting $g_2=1$.

\subsubsection{Type $\I_4^*$}

By Table~\ref{tab2}, the Enriques torsors form the family given by
\[
y^2 = tx^4 + t   x^2 + t^2  x + t^3c_1^4.
\]
Here we cannot rescale $t$ without altering one of the four normalized summands, so the family of Enriques surfaces depends on two moduli (the coefficients of $c_1$).

\subsubsection{Type $\II^*$}

Table~\ref{tab2} leads to the following $2$-dimensional family of torsors:
\[
y^2 = tx^4 + t^2  x + t^3c_1^4.
\]
After rescaling $x,t,y$, this simplifies to the $1$-dimensional family of Enriques surfaces
\[
y^2 = tx^4 + t^2  x + \gamma t^3(t+1)^4 \quad (\gamma\in k^\times).
\]

\subsubsection{Conclusion of the proof}

We recorded one exception to the moduli count from Section~\ref{ss:explicit}, exactly as stated in Theorem~\ref{thm2}. 

\begin{remark}
The results of Theorem~\ref{thm2} agree not only with Ogg--Shafarevich theory, but also with the predictions in the context of a conjecture of W.\ Lang (\textit{cf.} \cite[Section~4.8]{CDL}).
\end{remark}

\begin{remark}
The subtleties concerning moduli counts reappear when considering special subfamilies where reducible fibres are assumed to be multiple; see Section~\ref{ss:numer}.
\end{remark}

\section{Automorphisms of quasi-elliptic Enriques surfaces}
\label{s:aut}

It is a classical problem to investigate automorphisms of Enriques surfaces -- especially finite automorphism groups and those which act trivially on cohomology (or on $\Num$).  Finite automorphism groups will be studied in Section~\ref{s:finite}, so here we focus on numerically or cohomologically trivial automorphisms.  Over $\mathbb C$, the solution for these from \cite{MN} was later corrected in \cite{Mukai}.  In positive characteristic, there are extensive results in \cite{DM, DM2}, but the order $3$ case was left open, and the results also depended on the classification of Enriques surfaces with finite automorphism groups which we are about to complete with Theorem~\ref{thm3}.  To prepare for this, we will give further general results on automorphisms of quasi-elliptic surfaces in Section~\ref{ss:auts}, but first we give a quick application of the results culminating in Table~\ref{tab2}.

\subsection{Numerically trivial automorphisms}
\label{ss:numer}

Our first application concerns numerically trivial automorphisms of classical Enriques surfaces in characteristic $2$.  By \cite[Corollary~7.8]{DM2}, Enriques surfaces with non-trivial numerically trivial automorphisms admit quasi-elliptic fibrations with very specific multiple fibres, namely
\[
\I_4^* \quad \text{or} \quad \III^* + \III \quad \text{or} \quad 2 \times \I_4^* \quad \text{or} \quad \I_2^* + \III. 
\]
Here we can work out the surfaces explicitly, using Table~\ref{tab2}, and thus confirm the moduli dimensions which were given partly conditionally in \cite[Table 5]{DM2}.

\begin{proposition}
\label{prop:numer}
Let $S$ be a classical Enriques surface admitting a non-trivial numerically trivial automorphism.  Then $S$ admits a quasi-elliptic fibration, given in the following table, and moves in a family of specified dimension.
$$
\renewcommand{\arraystretch}{1.1}
\begin{array}{c|c|c}
  \text{Double fibres} & \text{Equation} & \text{Moduli}\\
  \hline
&& \\[\dimexpr-\normalbaselineskip+2pt]
\I_4^* + \II
&
y^2 = tx^4 + t^3  x^2 + t^5 x + t^3c_1^4
&
2
\\
\III^* + \III
&
y^2 + t^3 y = tx^4 + \gamma t^3(1+t)^4
&
1\\
2 \times \I_0^*
&
y^2 = tx^4 + \alpha t^3 x^2 + t^4 x + \gamma t^3(1+t)^4
&
2
\\
\I_2^* + \III
&
y^2 + (t+1)^3  y = tx^4 + t^3c_1^4
&
2
\end{array}
$$
\end{proposition}

\begin{proof}
We use the normal forms from Table~\ref{tab2}.  The configuration of multiple fibres can be translated into $g_2=t$ for the first three families and $g_2=t+1$ for the fourth.  For the second and third family, we can further rescale~$t,x,y$ to normalize $c_1$, resulting in the given equations and moduli counts.
\end{proof}

\subsection{Automorphisms preserving a quasi-elliptic fibration}
\label{ss:auts}

For independent use, we discuss the shape of automorphisms of quasi-elliptic Enriques surfaces further, building on Lemma~\ref{lem:iso2}.

\begin{lemma}
\label{lem:phi}
Let $\varphi$ be an automorphism of an Enriques surface $S$ which preserves some quasi-elliptic fibration.  Then $\varphi$ preserves some fibre of \eqref{eq:general}, say at $t=\infty$, and is given by
\begin{eqnarray}
\label{eq:phi0}
(x,y,t) \longmapsto (\beta x + b_2, \delta y + \sqrt\gamma x^2+d_3x + d_5, \alpha t + \gamma),
\end{eqnarray}
with scalars $\alpha, \beta, \gamma, \delta$ $($non-zero except possibly for $\gamma)$ and polynomials $b_2$, $d_3$, $d_5$ in $k[t]$ of degree bounded by the index.
\end{lemma}

\begin{proof}
The induced action of $\varphi$ on the base $\PP^1$ of the quasi-elliptic fibration has a fixed point, so $\varphi$ preserves the corresponding fibre.  By M\"obius transformation, we can map this fibre to $t=\infty$ and deduce the claimed action on $\PP^1$ with $\alpha\in k^\times$.  The rest is Lemma~\ref{lem:iso2} with the addition that $d_1=\sqrt\gamma$, as can be seen by comparing the coefficient of $x^4$ before and after the M\"obius transformation.
\end{proof}

One can also take the multiple fibre(s) into consideration: in the supersingular case, it is obviously also fixed by $\varphi$; in the classical case, either the multiple fibres are interchanged (if $\gamma\neq 0$), or each is preserved (whence $\gamma=0$ as either the multiple fibres are located at $0, \infty$, or $\varphi$ acts as identity on the base).  We consider one of these special cases in more detail. 

\begin{lemma}
\label{lem:aut2}
Assume that the automorphism $\varphi$ preserves a quasi-elliptic fibration and acts on the base $\PP^1$ with at least two fixed points which we locate at $0, \infty$.  Then it takes the shape
\begin{eqnarray}
\label{eq:phi2}
(x,y,t) \longmapsto (\beta x + b_2, \delta y + d_5, \alpha t)
\end{eqnarray}
for some non-zero scalars $\alpha, \beta, \delta$.  Moreover, if $\alpha\neq 1$, then the polynomials $a_1, a_2, g_2$ in \eqref{eq:general} are all monomial $($or zero for $a_1, a_2)$.
\end{lemma}

\begin{proof}
We may start with the shape of $\varphi$ given by \eqref{eq:phi0}, but then with two or more fixed points on $\PP^1$, we have $\gamma=0$ as stated.  We now substitute \eqref{eq:phi0} into \eqref{eq:general} and consider the coefficient of $x^2$.  Decomposing into even and odd terms, we find that $d_3\equiv 0$.  This proves that $\varphi$ takes the shape of \eqref{eq:phi2}.

Now assume that $\alpha\neq 1$.  As argued above, this locates the multiple fibres at $t=0$, say, and $\infty$ in the classical case.  That is, up to normalizing, $g_2=t$, resp.\ $g_2=t^2$; in particular, $g_2$ is monomial, as claimed.  Comparing further coefficients upon the substitution, we  subsequently deduce 
\begin{itemize}
\item
from the coefficient of $y$ that $a_1$ is monomial or zero,
\item
from the coefficient of $x$ that $a_2$ is monomial or zero.
\end{itemize}
This completes the proof of Lemma~\ref{lem:aut2}
\end{proof}

One can retrieve further information from studying the constant term of the normal form and then imposing the invariance of the normal form under $\varphi$, but we will only pursue this for one example from Proposition~\ref{prop:numer} and for some special cases needed to prove Theorems~\ref{thm3} and~\ref{thm:ct}.

\begin{example}
Lemma~\ref{lem:aut2} directly applies to the numerically trivial automorphisms on the classical Enriques surfaces from Proposition~\ref{prop:numer} because they automatically preserve reducible fibres (and hence all multiple fibres).

For instance, for the fourth example from Proposition~\ref{prop:numer} with three reducible fibres, we deduce that any numerically trivial automorphism $\varphi$ acts trivially on the base, so $\alpha=\beta=\delta=1$ in \eqref{eq:phi2}.  But then, comparing coefficients upon substitution, we find that $\deg(b_2)\leq 1$ and $\deg(d_5)\leq 3$, eventually leading to the following three cases with $\varphi\neq\mathrm{id}$:
\[
(b_2,d_5) =
(0,(1+t)^3), \; (1+t, 1+t^2), \; (1+t, t+t^3).
\]
The first two involutions permute the components of the simple fibre of type $\III$ at $t=0$, so they are not numerically trivial.  Meanwhile, going through the resolution of singularities, the third involution is checked to be numerically trivial.  Note that by \cite[Corollary~7.2]{DM}, it is not cohomologically trivial.
\end{example}

\begin{remark}
We note again that the surfaces given by \eqref{eq:ct} have infinite automorphism group.
\end{remark}

\section{Finite automorphism groups of Enriques surfaces in characteristic 2}
\label{s:finite}

Enriques surfaces with finite automorphism groups are notorious because a general Enriques surface has infinite automorphism group (contrary to the case of K3 surfaces, for instance).  The property of having finite automorphism group can be described combinatorially, namely in the incidence graph $\Gamma$ of smooth rational curves -- which in particular has to be finite.  This was employed to give a full classification over $\mathbb C$ by Kond\=o \cite{Kondo-aut} and Nikulin \cite{Nikulin}, and in odd characteristic by Martin \cite{Martin}.  In fact, Martin's work also covers the case of singular Enriques surfaces, while all possible graphs $\Gamma$ have been determined for the cases of classical and supersingular Enriques surfaces in \cite{KKM}.  In \textit{op.~cit.}, the authors also provide examples for each $\Gamma$, but no uniqueness or irreducibility of moduli is claimed (except for the extra-special cases covered in \cite{Salomonsson}).  In particular, the precise finite automorphism groups could not be classified in most cases.  Theorem~\ref{thm3} remedies this based on our normal form arguments.

\subsection{Overall idea}

Given a graph $\Gamma$ of smooth rational curves, we identify a divisor of Kodaira type which induces a quasi-elliptic fibration on the Enriques surface.  In particular, this determines the shape of the relative Jacobian, so Ito's equations in \cite{I} can be translated back to our Enriques surfaces using Theorem~\ref{thm2}, in particular, using the explicit calculations from Section~\ref{ss:explicit}.  Ideally, this quasi-elliptic fibration already fixes all the curves in~$\Gamma$, but in three cases we also have to consider a second fibration to verify the findings from \cite{KKM} (see Sections~\ref{ss:62} and~\ref{ss:VIII}).  In particular, this will prove that the examples and automorphism groups in \cite{KKM} are essentially complete, and thus verify Theorem~\ref{thm3}.

In what follows, we go through all possible graphs $\Gamma$ one by one; throughout, we employ the notation and findings of \cite{KKM}.

\subsection{$\boldsymbol{\Gamma = \tilde E_7 + \tilde A_1^{(1)}}$}
\label{ss:711}

These Enriques surfaces admit a quasi-elliptic fibration with reducible fibres of types III$^*$ (multiple) and III (simple).  Locating these fibres at $t=0$ and $t=\infty$, respectively, we see that Table~\ref{tab2} leads to the normal form
\begin{eqnarray}
\label{eq:72-1}
y^2 + t^3g_1^2 y = tx^4 + t^3 c_1^4, 
\end{eqnarray}
where $\deg(g_1)=1$ (with classical Enriques surfaces if $t\nmid g_1$ and supersingular Enriques surfaces otherwise).  Hence we normalize $g_1=t+a, c_1= b(t+1)$ such that $b\neq 0$.  The $2$-dimensional classical family is given by $a\neq 0$, and the $1$-dimensional supersingular family occurs at $a=0$.  This confirms that the examples of \cite{KKM} are complete, with anticipated moduli dimensions, and so are the possible automorphism groups computed.

\subsection{$\boldsymbol{\Gamma = \tilde E_7 + \tilde A_1^{(2)}}$}
\label{ss:712}

The Enriques surfaces with this graph also admit a quasi-elliptic fibration with reducible fibres of types III$^*$ and III, but contrary to Section~\ref{ss:711}, the special graph $\Gamma = \tilde E_7 + \tilde A_1^{(2)}$ implies that  the III fibre is also multiple.  Thus this only concerns the classical case and leads to the second family from Proposition~\ref{prop:numer} (given by $g_1=1$ in terms of \eqref{eq:72-1}), again confirming \cite{KKM}.

\subsection{$\boldsymbol{\Gamma = \tilde E_8}$}

There is a quasi-elliptic fibration with multiple II$^*$ fibre (which we locate at $t=0$).  From Table~\ref{tab2}, we obtain a $2$-dimensional family of torsors given by
\begin{eqnarray}
\label{eq:711}
y^2 = tx^4 + t^4g_1^2 x + t^3 c_1^4 \quad (t\nmid g_2).
\end{eqnarray}
We derive a $1$-dimensional family of classical Enriques surfaces where $t\nmid g_1$ (whence we normalize $g_1=\beta$ (with $\beta\neq 0$), $c_1=1+t$), and a single supersingular Enriques surface at $g_1=t$.  This is again in perfect agreement with \cite{KKM}.

\subsection{$\boldsymbol{\Gamma = \tilde D_8}$}

This graph induces a quasi-elliptic fibration with a multiple fibre of type I$_4^*$ which we locate at $t=0$.  The $2$-dimensional family of classical Enriques surfaces thus appears in Proposition~\ref{prop:numer}; the $1$-dimensional family of supersingular Enriques surfaces is obtained from Table~\ref{tab2} by setting $g_2=t^2$ and normalizing $c_1=\gamma\in k^\times$ by the M\"obius transformation.  Then the shape of \eqref{eq:general} can be restored by an easy transformation in $x$ and $y$, resulting in the normal form (with $\beta=\gamma^4$)
\[
y^2 = tx^4+ t^5x^2 + t^8x +\beta t^3.
\]

\subsection{$\boldsymbol{\Gamma = \tilde D_4 + \tilde D_4}$}

These Enriques surfaces admit a quasi-elliptic fibration with two multiple fibres of type I$_0^*$.  Hence they are given by the $2$-dimensional family of classical Enriques surfaces in Proposition~\ref{prop:numer}.

\subsection{$\boldsymbol{\Gamma =  \text{VII}}$}

Among all possible configurations of smooth rational curves on classical and supersingular Enriques surfaces with finite automorphism group, this graph is singled out by the property that it only induces elliptic fibrations.  It follows from the classification of the fibrations in \cite[Appendix~A.1]{KKM} that the universal cover has 12 $A_1$ singularities, so the minimal resolution is the supersingular K3 surface of Artin invariant $\sigma=1$.  Its Enriques quotients (classical and supersingular) have been determined in \cite{Kondo}.  In particular, this confirms the findings of \cite{KKM}.

\subsection{$\boldsymbol{\Gamma = \tilde E_6 + \tilde A_2}$}
\label{ss:62}

This type was covered in \cite{Salomonsson}, but we decided to include an independent proof for completeness; it also serves as a good illustration of the interplay of our normal forms with concrete linear systems.

\subsubsection{Classical case}
\label{sss:cl}

In the classical case, these surfaces admit a quasi-elliptic fibration with multiple fibres of type $\III$ at $t=0$ and $\II$ at $t=\infty$, plus a simple fibre of type $\III^*$ at $t=1$.  Using Table~\ref{tab2}, we compute the equation
\begin{eqnarray}
\label{eq:E6A2c}
S: \;\; y^2 + (t+1)t^2y = tx^4 + t^3 x^2 + (t+1)t^4 x + at^3 (t+b)^4, 
\end{eqnarray}
where $ab\neq 0$.  This has curve of cusps $C$ at $\infty$, making for the following graph of smooth rational curves together with the fibre components:

\begin{figure}[ht!]
\setlength{\unitlength}{.4in}
\begin{picture}(7,4.5)(0,-.5)
\thinlines

\put(-.2,-.5){$\Theta_1$}
\put(.8,-.5){$\Theta_2$}
\put(1.8,-.5){$\Theta_3$}
\put(2.8,-.5){$\Theta_4$}
\put(3.8,-.5){$\Theta_3'$}
\put(4.8,-.5){$\Theta_2'$}
\put(5.8,-.5){$\Theta_1'$}

\put(3.2,.86){$\Theta_0$}
\put(3.2,1.86){$C$}
\put(3.2,2.86){$C_0$}
\put(3.2,3.86){$C_1$}

\put(0,0){\line(1,0){6}}
\multiput(0,0)(1,0){7}{\circle*{.15}}

\put(3,0){\line(0,1){3}}
\multiput(3,1)(0,1){4}{\circle*{.15}}

\put(2.925,3){\line(0,1){1}}
\put(3.075,3){\line(0,1){1}}

\end{picture}
\caption{Smooth rational fibre components and curve of cusps for $f$}
\label{Fig:E6A2}
\end{figure}
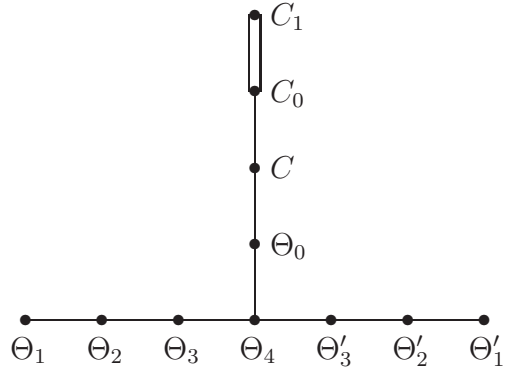

Note the divisor $D= C + 2\Theta_0 + 3\Theta_4 + 2 (\Theta_3+\Theta_3')+\Theta_2+\Theta_2'$ of Kodaira type $\IV^*$ central to the diagram.  We continue by exhibiting the linear system $|2D|$; by inspection of the diagram, we see that this will induce an elliptic fibration
\[
f'\colon S \lra \PP^1
\]
with multiple fibre of type $\IV^*$, nodal bisections $C_0, \Theta_1, \Theta_1'$ and another reducible fibre containing the curve $C_1$.

To make all of this explicit, we resolve the singularity in the fibre at $t=1$ as in the proof of Proposition~\ref{prop:ade}.  First we substitute $a=c^4$ and change variables $y=(t+1)y' + x^2 + tx + c^2t(t + b)^2$, amounting to the blow-up along $\Theta_0$ from Section~\ref{ss:1st}.  This produces the 4-fold fibre component
$$\Theta_4=(t+1=x + \gamma = 0),$$ 
where $\gamma^2 = (bc + c + 1)(b + 1)c$. 
After the  change of variables $x=x'+\gamma$,
the remaining pairs of fibre components $\Theta_i, \Theta_i'$ 
of multiplicity $i=1,2,3$  are successively uncovered in the three blow-ups
$t=x'^it_i+1$.
This shows that the function
\begin{eqnarray}
\label{eq:w}
w = \dfrac{(x+t\gamma)^2}{t(t+1)^2} = \dfrac{(1+\gamma t_1)^2}{(x't_1+1)t_1^2}
\end{eqnarray}
has pole divisor exactly $2D$ outside the multiple fibres at $t=0, \infty$.  Meanwhile, on the two multiple fibres of the original quasi-elliptic fibration, the resolution of singularities as in Section~\ref{s:res} shows that $w$ is regular and non-constant on the fibre component met by the curve of cusps.  In particular, on $C_0$, it gives a bisection, and we have $w=c^2b^2$ on the other fibre component $C_1$ at $t=0$.  That is, $C_1$ is a component of the (reducible) fibre of $f'$ at $w=c^2b^2$.

\begin{lemma}
The general member of the family of Enriques surfaces given by \eqref{eq:E6A2c} does not have finite automorphism group.
\end{lemma}

\begin{proof}
For the surface to fall into the finite automorphism group case $\Gamma = \tilde E_6 + \tilde A_2$, the second reducible fibre of the fibration $f'$ at $w=c^2b^2$ has to have type $\I_3$ or $\IV$.  Otherwise, the Jacobian would have positive Mordell--Weil rank by \cite{Lang, MWL}; thus it would induce an infinite-order automorphism on $S$.  We therefore continue by calculating the Kodaira type of this very fibre, depending on $b$ and $c$.  To this end, we substitute \eqref{eq:w} for $x'$ in the equation corresponding to \eqref{eq2}.  Rescaling the coordinate $y'$ by $h^2/t_1^5$ for $h=b^2\alpha^2 t_1 + b\alpha t_1 + \alpha^2 t_1 +\alpha t_1 + 1$, where $\alpha = \sqrt{c/(b+1)}$, we obtain an equation
\[
y''^2 + wt_1h^2 y'' = r \quad \text{with}\ h\in k[t_1,w] \text{ of degree $6$ in } t_1.
\]
This is not quite visibly a curve of arithmetic genus~$1$, but writing $r=r_0^2+r_1$, where $r_1$ is odd with respect to $t_1$ or $w_1$, the coordinate change $y'' = hu + r_0$ gives a curve of arithmetic genus~$1$ over $k(w)$:
\begin{eqnarray}
\label{eq:w-spec}
u^2 + wt_1hu = wR \quad \text{with}\ R\in k[t_1,w] \text{ of degree $4$ in } t_1.
\end{eqnarray}
This has discriminant
\[
\Delta' = w^{12} (w^2 + w + c^4)(w + c^2b^2)^2
\]
 vanishing generally to order $2$ at $w=c^2b^2$; since we already know that the fibre at $w=c^2b^2$ is reducible as it contains the smooth rational component $C_1$, this verifies the Kodaira type $\I_2$.  The lemma follows.
\end{proof}

In order to determine the subfamily with finite automorphism group, it remains to check when the fibre at $w=c^2b^2$ degenerates to Kodaira type $\I_3$.  The discriminant $\Delta'$ acquires a triple root at $w=c^2b^2$ if and only if $c=b/(1+b)^2$.  One verifies that the fibre type is indeed $\I_3$ by checking that the fibre of \eqref{eq:w-spec} at $w=c^2b^2$ always contains exactly one singular point (a surface singularity of type $A_1$ where $h$ vanishes), but it becomes reducible on the given subfamily.  (Alternatively, both fibre types $\I_2$ and $\I_3$ follow generally from the classification of wild ramification of additive fibres in characteristic $2$ in \cite{SSc} as this implies that $\Delta'$ has vanishing order at least $4$ at any additive fibre.)  With fibres of type $\IV^*$ and $\I_3$, the Jacobian of $f'$ has $\mathrm{MW}(\Jac(f'))\cong\Z/3\Z$.  But then the symmetry imposed by the induced order $3$ automorphism on $S$ implies that the curves in Figure~\ref{Fig:E6A2} augmented by the two additional components of the $\I_3$ fibre of $f'$ produce exactly the graph of smooth rational curves for type $\Gamma = \tilde E_6 + \tilde A_2$ from \cite{KKM}.  This gives the claimed irreducible $1$-dimensional family of classical Enriques surfaces with finite automorphism group from Theorem~\ref{thm3}.

\subsubsection{Supersingular case}
\label{sss:ss}

In the supersingular case,
the above set-up degenerates to a quasi-elliptic fibration with a single multiple fibre 
of type $\III$ at $t=0$
and another reducible fibre, simple as before, of type $\III^*$ at $t=1$.
From Table~\ref{tab2}, we obtain the equation
\begin{eqnarray}
\label{eq:E6A2s}
S: \;\; y^2 + (t+1)t^4y = tx^4 + t^5 x^2 + (t+1)t^7 x + ct^3 (t+1)^4, 
\end{eqnarray}
where $c\neq 0$.  Note that we used a different normalization than in \eqref{eq:supersingular} because the given one turns out to be much more convenient as it automatically leads to a unique Enriques surface.  The graph of smooth rational curves in Figure~\ref{Fig:E6A2} is still standing and so is the divisor $D$ of Kodaira type $\IV^*$, but now the linear system $|2D|$ is generated by
\[
u = \frac{x^2+c(t + 1)^3t}{t^2(t+1)^2}.
\]
In terms of the induced fibration $f'$, the smooth rational curve $C_1$ is a component of the fibre at $u=c$.  The fibre contains a single singular point which is never a surface singularity for any $c\neq 0$; its projectivized tangent cone consists of a single line unless $c=1$.  We now combine Kodaira's and Tate's classifications of singular fibres, see \cite{K, Tat}, with an Euler--Poincar\'e characteristic reasoning taking into account the multiple fibre of type $\IV^*$ at $u=\infty$.  Presently this implies that the fibre can only have type $\III$ for $c\neq 0,1$.  But then $S$ inherits an automorphism of infinite order from $\Jac(f')$ by the Shioda--Tate formula, ruling out all values $c\neq 0,1$.

On the other hand, for $c=1$, the fibre acquires three smooth rational irreducible components which meet in the singular point (with different tangent directions).  This verifies that the fibre has type $\IV$.  In particular, we have $\mbox{MW}(\Jac(f')) = \Z/3\Z$ by \cite{MWL}, and the symmetry imposed by the induced automorphism on $S$ exactly leads to the required graph $\Gamma$ of smooth rational curves.  Thus, at $c=1$, the Enriques surface $S$ has finite automorphism group.

\subsubsection{Conclusion for type $\boldsymbol{\Gamma = \tilde E_6 + \tilde A_2}$}
\label{sss:62-aut}

We have not explicitly identified our Enriques surfaces with type $\Gamma$ with those from \cite{KKM}.  For the supersingular case, it follows from the uniqueness of the surface in Section~\ref{sss:ss} that it has to agree with the one from \cite{KKM}.  Thus  the automorphism group also agrees (but it can also be computed directly using Lemma~\ref{lem:aut2}).

In the classical case, our $1$-dimensional family from Section~\ref{sss:cl} is irreducible and unique and thus contains the family from \cite{KKM}.  Instead of identifying the two families explicitly, we proceed by computing the automorphism group for the family from Section~\ref{sss:cl}.

The symmetry group of the graph $\Gamma$ is $\mathfrak S_3$, and these automorphisms are always induced from the Mordell--Weil groups of the Jacobians of genus~$1$ fibrations on $S$ (it suffices to consider those fibrations denoted by $f$ and $f'$ in Section~\ref{sss:cl}).

It remains to compute the numerically trivial automorphisms.  Necessarily, they preserve any genus~$1$ fibration on $S$, and they fix any reducible fibre (componentwise).  We can thus apply Lemma~\ref{lem:aut2} to the quasi-elliptic fibrations from \eqref{eq:E6A2c} with $a=b^4/(1+b)^8$ (family \eqref{c4} in Theorem~\ref{thm:c-eqn}) and consider the automorphism $\varphi$ given by \eqref{eq:phi2}.  Since $a_1$ is not monomial, Lemma~\ref{lem:aut2} implies that $\alpha=1$.  Substituting \eqref{eq:phi2} and comparing the coefficients of $y$ and $x^2$ gives $\delta=\beta=1$.  This leaves the constant coefficient, where vanishing orders at $0$ and at $\infty$ yield
 \[
 b_2 = b't,\quad d_5 = dt^2+d't^3 \quad (b',d,d'\in k).
 \]
 The three constants satisfy a system of three equations which is seen to have exactly two solutions:
\[
(b',d,d') \in \{(0,0,0), \; (0,1,1)\}.
\] 
The automorphism $\varphi$ corresponding to the second solution is an involution which acts non-trivially on $\Num(S)$ as it interchanges the two branches of the $\III^*$ fibre at $t=1$ (it is induced by translation by the $2$-torsion section on $\mathrm{MW}(\Jac(f))$).  Thus we conclude that $\Aut_{nt}(S)=\{\mathrm{id}\}$, confirming the examples from \cite{KKM}.

\subsection{$\boldsymbol{\Gamma = \text{VIII}}$}
\label{ss:VIII}

By \cite{KKM}, Enriques surfaces of this type admit a quasi-elliptic fibration with two multiple fibres of type $\III$ and a simple fibre of type $\I_2^*$.  By Table~\ref{tab2}, the normal form is given as
\begin{eqnarray}
\label{eq:VIII}
S: \;\; y^2 + t^2 (t+1) y = tx^4 + t^3 (at+b)^4 \quad (a,b\in k^*).
\end{eqnarray}
The diagram of reducible fibre components enriched by the curve of cusps $C$

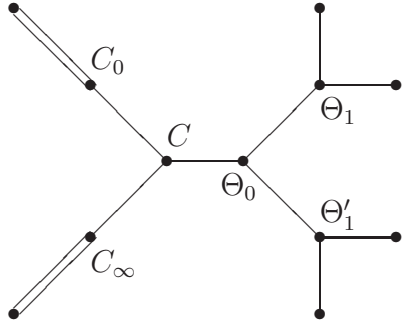
\begin{figure}[ht!]
\setlength{\unitlength}{.4in}
\begin{picture}(5,4.4)(0,-0.2)
\thinlines

\multiput(0,0)(4,0){2}{\circle*{.15}}
\multiput(0,4)(4,0){2}{\circle*{.15}}
\multiput(4,3)(1,0){2}{\circle*{.15}}
\multiput(4,1)(1,0){2}{\circle*{.15}}
\multiput(2,2)(1,0){2}{\circle*{.15}}
\multiput(1,1)(0,2){2}{\circle*{.15}}

\put(2,2.2){$C$}
\put(2.7,1.55){$\Theta_0$}
\put(4,1.2){$\Theta_1'$}
\put(4,2.55){$\Theta_1$}
\put(1,3.2){$C_0$}
\put(1,.55){$C_\infty$}

\put(1,1){\line(1,1){1}}
\put(1,3){\line(1,-1){1}}
\put(2,2){\line(1,0){1}}
\put(3,2){\line(1,1){1}}
\put(3,2){\line(1,-1){1}}
\put(4,1){\line(1,0){1}}
\put(4,1){\line(0,-1){1}}
\put(4,3){\line(1,0){1}}
\put(4,3){\line(0,1){1}}

\put(0,0.075){\line(1,1){1}}
\put(0.075,0){\line(1,1){1}}

\put(0,3.925){\line(1,-1){1}}

\put(0.075,4){\line(1,-1){1}}
\end{picture}
\end{figure}

\noindent
features a divisor $D=C_0+C_\infty+2C+2\Theta_0+\Theta_1+\Theta_1'$ of Kodaira type $\I_1^*$.  It follows that $|2D|$ induces an elliptic fibration
\[
f'\colon S\lra\PP^1
\]
 with multiple fibre of type $\I_1^*$  and at least four smooth rational bisections
 as well as two smooth rational $4$-sections.
Explicitly, the fibration can be exhibited through the elliptic parameter
\[
u = \frac{(x^2 + tx + a^2t^3 + at^2 + bt^2 + b^2t)^2}{t^3(t+1)^2}
\]
which has pole divisor exactly $2D$.  The fibration $f'$ visibly has a double fibre at $u=0$, with smooth support.  In addition to the double $\I_1^*$ fibre at $u=\infty$, there are generally two reducible fibres, of Kodaira type $\I_2$ each, at $u=a^2$ and $u=b^2$.  It follows that $S$ has infinite automorphism group (induced from $\mbox{MW}(\Jac(f'))$ unless $a=b$.  To see that this exactly recovers the family
$$
   v^2 = tu^4 + at^2u^3 + at^3(t+1)^2u + t^3(t+1)^4 \quad (a \neq 0)
$$
from \cite[Appendix A.2(3)]{KKM}, one may apply the following birational transformation with  $c ={1}/{a^2}$: 
$$
x = {\frac{(t+1)v}{\sqrt{a}(u^2 + t(t+1)^2)}},\quad
y = {\frac{t^3(t+1)^3}{u^2 + t(t+1)^2}}.
$$

\subsection{Proof of Theorem~\ref{thm3}}
\label{ss:concl}

With all possible graphs of smooth rational curves covered, each type yields irreducible families of classical and supersingular Enriques surfaces of the expected dimensions.  These families agree with the examples exhibited in \cite{KKM}, or they contain the examples from \cite{KKM} for type $\Gamma=\tilde E_6+\tilde A_2$, but then we checked in Section~\ref{sss:62-aut} that the automorphism groups agree anyway.  The computations of the numerically (resp.\ cohomologically) trivial automorphism groups from \cite{KKM} remain valid (confirmed by \cite{DM, DM2}).  This completes the proof of Theorem~\ref{thm3}.  \qed

\subsection{Equations}

For the convenience of the reader, we collect normal forms for all quasi-elliptic Enriques surfaces with finite automorphism group.

\begin{theorem}
\label{thm:c-eqn}
The normal forms of quasi-elliptic classical Enriques surfaces with finite automorphism group and their numbers $n$ of moduli are given as follows:

\small{
\begin{table}[H]
  \centering
{\setlength{\extrarowheight}{1pt}
\begin{tabular}{|c|c|l|c|}
\hline
     &\rm Type & \quad {\rm Normal form} &  $n$
\\ \hline
{\rm (c1)} & $\tilde{E_8}$ & $y^2= tx^4 + at^5x + t^3(1+t)^4$ \quad $(a \neq 0)$ & 1
\\ \hline
\rule[0pt]{0pt}{\heightof{$\tilde{A_1}^{(1)}$}+.5ex}{\rm (c2)} & $\tilde{E_7}+\tilde{A_1}^{\!(1)}$ &$y^2 + at^3y = tx^4 + bt^5x + t^3(1+t)^4$ \quad $(a\neq 0, b\neq 0)$ & 2
\\ \hline
\rule[0pt]{0pt}{\heightof{$\tilde{A_1}^{(1)}$}+.5ex}{\rm (c3)} &$\tilde{E_7}+\tilde{A_1}^{\!(2)}$ & $y^2 + at^3y = tx^4 + t^3(1 + t)^4$ \quad $(a\neq 0)$ & 1
\\ \hline
\refstepcounter{foo}\otherlabel{c4}{c4}{\rm (c4)} & $\tilde{E}_6+\tilde{A_2}$ & $y^2 + (t+1)t^2y=tx^4 +t^3x^2 +(t+1)t^4x + t^3(t+b)^4b^4/(b+1)^8$ & 1 \\
&& 
\hfill$(b \neq 0,1)$  & \\
 \hline
{\rm (c5)} & $\tilde{D_8}$ & $y^2 = tx^4+ at^3x^2 + bt^3x + t^3(1+t)^4$\quad $(b \neq 0)$ &2
\\ \hline
{\rm (c6)} & $\tilde{D}_4 + \tilde{D}_4$  & $y^2 = tx^4 + at^3x^2  + bt^4x + t^3(1 + t)^4$\quad $(b \neq 0)$  &2
\\ \hline
{\rm (c8)} & {\rm VIII} & 
$y^2 + t^2(t+1)y = tx^4 + ct^3(1+t)^4$ \quad  $(c \neq 0)$  &1
\\ \hline
\end{tabular}}
\end{table}
}
\end{theorem}

\begin{theorem}
\label{thm:s-eqn}
The normal forms of quasi-elliptic supersingular Enriques surfaces with finite automorphism group and their numbers $n$ of moduli are given as follows:

\begin{table}[!htb]
\centering
{\setlength{\extrarowheight}{1pt}
\begin{tabular}{|c|c|l|c|}
\hline
     &\rm Type & \quad {\rm Normal form} &  $n$
\\ \hline
{\rm (s1)} & $\tilde{E_8}$ & $y^2= tx^4 + x + t^7$  & 0
\\ \hline
\rule[0pt]{0pt}{\heightof{$\tilde{A_1}^{(1)}$}+.5ex}{\rm (s2)} & $\tilde{E_7}+\tilde{A_1}^{\!(1)}$ & $y^2 + y = tx^4 + ax + t^7$ \quad $(a\neq 0)$&1
\\ \hline
 \refstepcounter{foo}\otherlabel{s3}{s3}{\rm (s3)} 
   & $\tilde{E}_6+\tilde{A_2}$ & 
$y^2 + t^4 y = tx^4 + t^3$ & 0
\\  \hline

{\rm (s4)} & $\tilde{D_8}$ & $y^2 = tx^4+ tx^2 + ax +t^7$\quad $(a \neq 0)$ &1
\\ \hline
\end{tabular}}
\end{table}
\end{theorem}

\begin{proof}
All equations have been derived in equivalent form before, only \eqref{s3} meriting some explanation.  The given equation results from $c=1$ in \eqref{eq:E6A2s} by moving the $\III^*$ fibre to $\infty$ combined with an easy coordinate transformation used to recover the general shape of \eqref{eq:general}.
\end{proof}

\section{Proofs of Theorem~\ref{thm:ct} and Corollary~\ref{cor:ct}}

\subsection{Proof of Theorem~\ref{thm:ct}}

Let $\varphi$ be a cohomologically trivial automorphism of odd order $n>1$ acting on an Enriques surface $S$.  Until this paper, the only examples originated from the supersingular Enriques surfaces with finite automorphism groups pioneered in \cite{KKM}.  By Theorem~\ref{thm3}, the examples from \cite{KKM} are complete, so there are no more cohomologically trivial automorphisms of odd order $n>1$ acting on Enriques surfaces with finite automorphism group.  But then it follows from \cite[Section~7]{DM} that the only case left is $n=3$.  Concretely, by \cite[Proposition~7.9]{DM}, the Enriques surface $S$ admits a configuration of rational curves amounting to a quasi-elliptic fibration with two fibres of type I$_0^*$, exactly one of which is multiple.  Moreover, the cohomologically trivial action implies that $\varphi$ preserves any genus~$1$ fibration on $S$; more precisely, by \cite[Lemma~7.5]{DM}, $\varphi$ acts non-trivially on the base of the fibration.

By Table~\ref{tab2}, we may assume that $g_2=t^2$ and $S$ is given by
\begin{eqnarray}
\label{eq:2D4-3}
S: \;\; y^2 = tx^4 + \alpha t^5 x^2 + t^7 x + t^3 c_1^4 \quad (\alpha\in k).
\end{eqnarray}
Since we know that $t\nmid c_1$, we can normalize $c_1=1+\eta t$ to obtain a $2$-dimensional family of Enriques surfaces.  By assumption, $\varphi$ preserves the curve of cusps and each irreducible fibre component of the two I$^*_0$ fibres.  By Lemma~\ref{lem:aut2}, the automorphism is given by
\begin{eqnarray}
\label{eq:phi3}
\varphi\colon 
(x,y,t) \longmapsto (\beta x + b_2, \delta y + d_5, \zeta t), 
\end{eqnarray}
where $\zeta$ is a primitive third root of unity.  Upon substituting, we compare coefficients of the constant term to deduce that $b_2=\beta't^2 \, (\beta'\in k)$ and $d_5\equiv 0$.  But then the coefficients of $t^3$ and $t^7$ combine for $\zeta^3t^3(1+\zeta\eta t)^4$.  This is proportional to the original term $t^3c_1^4$ if and only if $\eta=0$, leading to the family of supersingular Enriques surfaces \eqref{eq:ct} from Theorem~\ref{thm:ct}.  Finally, the coefficient at $t^9$ gives
\begin{eqnarray}\label{eq:beta'}
\beta'(\beta'^3+\alpha\zeta\beta'+1) = 0.
\end{eqnarray}
Solving for \eqref{eq:ct} to be preserved by $\varphi$ leads to
\[
\beta=\zeta^2, \quad \delta=1. 
\]
Note that for choices $\beta'\neq 0$ in \eqref{eq:beta'}, the order of $\varphi$ is even, so we are left with $\beta'=0$, giving the order $3$ automorphism
\begin{eqnarray}
\label{eq:zeta^3}
\varphi\colon (x,y,t) \longmapsto (
\zeta^2 x, y,
\zeta t)
\end{eqnarray}
displayed in Theorem~\ref{thm:ct}.  It remains to prove that the action of $\varphi$ on the Enriques surfaces given by \eqref{eq:ct} is indeed cohomologically trivial.  This is easily checked by resolving the singularities as explained in Sections~\ref{s:ade} and~\ref{s:res}. 

\begin{remark}
The same approach shows that the special Enriques surface at $\alpha=0$ admits an automorphism $\psi$ of order $9$, with $\zeta$ a primitive ninth root of unity.  However, $\psi$ is verified to permute three of the simple fibre components at $t=0$, and considering $\psi^3$ reduces to the automorphism $\varphi$ from \eqref{eq:zeta^3}.

\end{remark}

\subsection{Proof of Corollary~\ref{cor:ct}}

The results of Corollary~\ref{cor:ct} can be found in \cite{DM} with two restrictions:
\begin{enumerate}[(1)]
\item
The possible lists of numerically trivial subgroups of the automorphism group in \eqref{c:ct-3} and \eqref{c:ct-4} rely on the validity of the computations of \cite{KKM} for all classical and supersingular Enriques surfaces with finite automorphism groups which we confirmed in Theorem~\ref{thm3}. 
\item
The appearance of $\Z/3\Z$ as a group of numerically trivial automorphisms depends on Theorem~\ref{thm:ct}.
\end{enumerate}
Thus Corollary~\ref{cor:ct} is proved in its entirety.
\qed


\newcommand{\etalchar}[1]{$^{#1}$}


\begin{thebibliography}{BHP\etalchar{+}04+++}

\bibitem[Art77]{Artin}
M.~Artin,
\emph{Coverings of the Rational Double Points in Characteristic $p$}, 
in: \emph{Complex Analysis and Algebraic Geometry}, pp. 11-22, Iwanami Shoten Publishers, Tokyo, 1977.

\bibitem[BHP\etalchar{+}04]{BHPV}
W.~Barth, K.~Hulek, C.~Peters and A.~Van de Ven,
\emph{Compact complex surfaces}, 2nd ed., 
Ergeb.\ Math.\ Grenzgeb.~(3), vol.~4, Springer-Verlag, Berlin, 2004.

\bibitem[BM76]{BM}
E.~Bombieri and D.~Mumford, \emph{Enriques' classification of surfaces in char.~$p$. III},  Invent.\ Math.\ {\bf 35} (1976), 197--232.
 
\bibitem[Cha80]{Chakiris}
 K.~Chakiris,
 \emph{Counterexamples to global Torelli theorem for certain simply connected surfaces},
 Bull.\ Amer.\ Math.\ Soc.\ (N.S.) {\bf 2} (1980), no.~2, 297--299.

\bibitem[Cos85]{Cossec}
F.\,R.~Cossec, 
\emph{On the Picard group of Enriques surfaces}, 
Math.\ Ann.\ {\bf 271} (1985), no.~4, 577--600.

\bibitem[CD89]{CD}
F.\,R.~Cossec and I.\,V.~Dolgachev, 
\emph{Enriques surfaces I}. 
Progr.\ Math., vol.~76, Birkh\"auser Boston, Inc., Boston, MA, 1989.

\bibitem[CDL22]{CDL}
F.\,R.~Cossec, I.\,V.~Dolgachev and C.~Liedtke, 
\emph{Enriques surfaces I}
(with an appendix by S.\ Kond\=o), version Sept.~21, 2022; updated version available at \url{https://dept.math.lsa.umich.edu/~idolga/EnriquesOne.pdf}

\bibitem[DK22]{DK}
I.\,V.~Dolgachev and S.~Kond\=o, \emph{Enriques surfaces II}, version Sept.~21, 2022; updated version available at \url{https://dept.math.lsa.umich.edu/~idolga/EnriquesTwo.pdf}. 

\bibitem[DM19]{DM}
I.\,V.~Dolgachev and G.~Martin, 
\emph{Numerically trivial automorphisms of Enriques surfaces in characteristic 2}, J.\ Math.\ Soc.\ Japan {\bf 71} (2019), no.~4, 1181--1200.

\bibitem[DM20]{DM2}
\bysame,
\emph{Automorphism groups of rational elliptic and quasi-elliptic surfaces in all characteristics},
Adv.\ Math.\ {\bf 400} (2022), Paper No.~108274.

\bibitem[Il79]{Illusie}
L.~Illusie, 
\emph{Complexe de de Rham-Witt et cohomologie cristalline},
Ann.\ Sci.\ \'Ecole Norm.\ Sup. ~(4) {\bf 12} (1979), no.~4, 501--661.

\bibitem[Ito94]{I}
H.~Ito, 
\emph{The Mordell-Weil groups of unirational quasi-elliptic 
surfaces in characteristic $2$}, 
Tohoku Math.\ J.~(2) {\bf 46} (1994), no.~2, 221--251.

\bibitem[KK18]{KaKo2}
T.~Katsura and S.~Kond\=o,
\emph{Enriques surfaces in characteristic $2$ with a finite group of automorphisms}, J.~Algebraic Geometry {\bf 27} (2018), no.~1, 173--202.

\bibitem[KKM20]{KKM}
T.~Katsura, S.~Kond\=o and G.~Martin,
\emph{Classification of Enriques surfaces with finite automorphism group in characteristic $2$},
Algebr.\ Geom.\ {\bf 7} (2020), no.~4, 390--459.

\bibitem[Kod63]{K}
K.~Kodaira,
\emph{On compact analytic surfaces I--III},
Ann.~of Math.~(2) {\bf 71} (1960), 111--152;
{\bf 77} (1963), 563--626; {\bf 78} (1963), 1--40.

\bibitem[Kon86]{Kondo-aut}
S.~Kond\=o, 
\emph{Enriques surfaces with finite automorphism group}, 
Japan J.~Math.\ (N.S.) {\bf 12} (1986), no.~2, 192--282.

\bibitem[Kon21]{Kondo}
\bysame, 
\emph{Classification of Enriques surfaces covered by the supersingular K3 surface with Artin invariant $1$ in characteristic $2$,}
J.~Math.\ Soc.\ Japan {\bf 73} (2021), no.~1, 301--328.

\bibitem[Lan88]{Lang}
W.~Lang, 
\emph{On Enriques surfaces in characteristic $p$. II}, Math.\ Ann.\ {\bf 281} (1988), no.~4, 671--685.

\bibitem[Lie15]{Liedtke}
C.~Liedtke, 
\emph{Arithmetic moduli and lifting of Enriques surfaces}, 
J.~reine angew.\ Math.\ {\bf 706} (2015), 35--65. 

\bibitem[Mar19]{Martin}
G.~Martin,
\emph{Enriques surfaces with finite automorphism group in positive characteristic}, Algebr.\ Geom.~{\bf 6} (2019), no.~5, 592--649. 

\bibitem[Muk10]{Mukai}
S.~Mukai, 
\emph{Numerically trivial involutions of Kummer type of an Enriques surface}, 
Kyoto J.~Math.\ {\bf 50} (2010), no.~4, 889--902.

\bibitem[MN84]{MN}
S.~Mukai and Y.~Namikawa, 
\emph{Automorphisms of Enriques surfaces which act trivially on the cohomology groups},
Invent.\ Math.\ {\bf 77} (1984), no.~3, 383--397.

\bibitem[Nik84]{Nikulin}
V.\,V.~Nikulin, 
\emph{On a description of the automorphism groups of Enriques surfaces},
Sov.\ Math., Dokl.\ {\bf  30} (1984), 282--285;
translation from Dokl.\ Akad.\ Nauk SSSR {\bf 277} (1984), 1324--1327. 

\bibitem[PP23]{PP}
G.~Pearlstein and C.~Peters,
\emph{A Remarkable Class of Elliptic Surfaces of Amplitude $1$  in Weighted Projective Space},
preprint \arXiv{2302.09358v2} (2023).

\bibitem[Que71]{Q1}
C.\,S.~Queen, 
\emph{Non-conservative function fields of genus one I}, 
Arch.\ Math.\ (Basel) {\bf 22} (1971), 612--623.

\bibitem[Que72]{Q2}
\bysame, 
\emph{Non-conservative function fields of genus one II}, 
Arch.\ Math.\ (Basel){\bf 23} (1972), 30--37.

\bibitem[RS76]{RS}
R.~Rudakov and I.\,R.~\v{S}afarevi\v{c},
\emph{Inseparable morphisms of algebraic surfaces}, 
Izv.\ Akad.\ Nauk SSSR Ser.\ Mat.\ \textbf{40} (1976), no.~6, 1269--1307, 1439. 

\bibitem[Sal03]{Salomonsson}
P.~Salomonsson,
\emph{Equations for some very special Enriques surfaces in characteristic two},
preprint  \arXiv{math/0309210} (2003). 

\bibitem[Sch19]{S-hom2}
M.~Sch\"utt, 
\emph{$\mathbb Q_\ell$-cohomology projective planes and Enriques surfaces in  characteristic two}, \'Epijournal de G\'eom.\ Alg\'ebrique {\bf 3} (2019), Art.~10. 

\bibitem[SS13]{SSc} 
M.~Sch\"utt and A.~Schweizer.
\emph{On the uniqueness of elliptic K3 surfaces with maximal singular fibre}, 
Ann.~Inst.~Fourier (Grenoble) {\bf 63} (2013), no.~2,  689--713.

\bibitem[SS19]{MWL}
M.~Sch\"utt and T.~Shioda.
\emph{Mordell--Weil lattices},
Ergeb.\  Math.\ Grenzgeb.~(3), vol.~70, Springer, Singapore, 2019.

\bibitem[Tat75]{Tat}
J.~Tate, \emph{Algorithm for determining the type of singular fibre in an elliptic pencil},
in: \emph{Modular Functions of One Variable IV} (Proc.\ Internat.\ Summer School, Univ.\ Antwerp, Antwerp, 1972), pp.~33--52, Lect.\ Notes in Math.\ vol.~476, Springer-Verlag, Berlin-New York, 1975.


\end{thebibliography}
\end{document}